\documentclass[11pt,letterpaper]{amsart}
 
\RequirePackage{fix-cm}
\usepackage{graphics, subcaption}
\usepackage{epsf,color,fancybox,epsfig,graphicx}

\usepackage[numbers,sort&compress]{natbib}

\usepackage{amssymb,latexsym,amsmath,amsfonts}
\usepackage{physics} 
\usepackage{amssymb,latexsym,amsmath,amsfonts,bm}
\usepackage[unicode,breaklinks=true,colorlinks=true,linkcolor=black,urlcolor=black,citecolor=black]{hyperref}
\usepackage{cleveref}
\usepackage{fourier}  
\usepackage{marginnote}
\usepackage[usenames,dvipsnames,svgnames,table]{xcolor}

\def\andand{\quad\text{and}\quad}
\def\withwith{\quad\text{with}\quad}
\def\BT{\text{BT}}
\def\fcomma{\,,}
\def\fperiod{\,.}
\def\fsemicolon{\,;}

\def\eold{\color{black}}

\def\Tf{T_{\text{f}}}

\def\indeq{\quad{}}

\def\lec{\lesssim}
\def\comma{ {\rm ,\quad{}} }

\newcommand{\buin}{\bu_{\text{in}}}

\newcommand{\bz}{\m{\bom$\zeta$\ubm}}
\newcommand{\bzn}{\m{\bom$\zeta$\ubm}^n}
\newcommand{\bznpo}{\m{\bom$\zeta$\ubm}^{n+1}}

\newcommand{\cpro}{c_{\text P}}
\newcommand{\cpCau}{c_{\text{Pa}}}
\newcommand{\cSob}{c_{\text S}}
\newcommand{\cCau}{c_{\text a}} 
\newcommand{\cSobCau}{c_{\text Sa}} 
\newcommand{\cCautwo}{c_{{\text a}_2}}

\definecolor{coloroooo}{rgb}{0.45,0.0,0.0}
\def\cole{\color{black}}

\definecolor{mygray}{rgb}{.6,.6,.6}
\chardef\forshowkeys=0
  \chardef\refcheck=0
  \chardef\showllabel=0
  \chardef\sketches=0

\ifnum\forshowkeys=1
  
  \usepackage[notref,notcite,color]{showkeys}
\fi
\ifnum\refcheck=1
  \usepackage{refcheck}
\fi
\ifnum\showllabel=1
 \def\llabel#1{\marginnote{\color{lightgray}\rm\small(#1)}[-0.0cm]\notag}
\else
\def\llabel#1{\notag}
\fi

\makeatletter
\newcommand{\opnorm}{\@ifstar\@opnorms\@opnorm}
\newcommand{\@opnorms}[1]{%
  \left|\mkern-1.5mu\left|\mkern-1.5mu\left|
   #1
  \right|\mkern-1.5mu\right|\mkern-1.5mu\right|
}
\newcommand{\@opnorm}[2][]{%
  \mathopen{#1|\mkern-1.5mu#1|\mkern-1.5mu#1|}
  #2
  \mathclose{#1|\mkern-1.5mu#1|\mkern-1.5mu#1|}
}
\makeatother

\def\colb{\color{black}}

\usepackage[english]{babel}

\usepackage{marvosym}

\newtheorem{theorem}{Theorem}[section]
\newtheorem{proposition}[theorem]{Proposition}
\newtheorem{lemma}[theorem]{Lemma}
\newtheorem{corollary}[theorem]{Corollary}
\theoremstyle{definition}
\newtheorem{definition}[theorem]{Definition}

\theoremstyle{remark}
\newtheorem{remark}[theorem]{Remark}

\numberwithin{equation}{section}

\newcommand{\vphi}{\varphi}
\newcommand{\ep}{\varepsilon}

\newcommand{\bom}{\boldmath}           
\newcommand{\m}{\mbox}                     
\newcommand{\ubm}{\unboldmath}        
\newcommand{\bu}{\m{\bom$u$\ubm}}  
\newcommand{\bnabla}{\m{\bom$\nabla$\ubm}}      
\newcommand{\buo}{\bu^1}                                       
\newcommand{\but}{\bu^2}										

\newcommand{\bn}{\m{\bom$n$\ubm}}  

\newcommand{\bv}{\m{\bom$v$\ubm}}

\newcommand{\pardt}{\partial_t}     

\newcommand{\pardy}{\partial_y}
\newcommand{\pardx}{\partial_x}              

\newcommand{\pardyy}{\partial^2_{y}}

\newcommand{\be}{\begin{equation}}
\newcommand{\ee}{\end{equation}}
\newcommand{\bea}{\begin{eqnarray}}
\newcommand{\eea}{\end{eqnarray}}
\newcommand{\beas}{\begin{eqnarray*}}
\newcommand{\eeas}{\end{eqnarray*}}

\newcommand{\Jreg}{J_\ep}
\newcommand{\Freg}{F_\ep}
\newcommand{\bureg}{\bu_\ep}

\begin{document}
\title[Local analyticity for the Euler equations]{On the local analyticity for the Euler equations}
\author{Igor Kukavica } \address{Department of Mathematics, University of  Southern California}
\email{kukavica@usc.edu}
 \author{Maria Carmela Lombardo} \address{Dipartimento di Matematica, Universit\`a di Palermo}
\email{mariacarmela.lombardo@unipa.it}
\author{Marco  Sammartino}\address{Dipartimento di Ingegneria, Universit\`a di Palermo}
\email{marco.sammartino@unipa.it}

\maketitle
\begin{abstract}
In this paper, we study the existence and uniqueness of solutions to the Euler equations with initial conditions that exhibit analytic regularity near the boundary and Sobolev regularity away from it. A key contribution of this work is the introduction of the diamond-analyticity framework, which captures the spatial decay of the analyticity radius in a structured manner, improving upon uniform analyticity approaches. We employ the Leray projection and a nonstandard mollification technique to demonstrate that the quotient between the imaginary and real parts of the analyticity radius remains unrestricted, thus extending the analyticity persistence results beyond traditional constraints. Our methodology combines analytic-Sobolev estimates with an iterative scheme which is nonstandard in the Cauchy-Kowalevskaya framework, ensuring rigorous control over the evolution of the solution. 
These results contribute to a deeper understanding of the interplay between analyticity and boundary effects in fluid equations. They might have implications for the study of the inviscid limit of the Navier-Stokes equations and the role of complex singularities in fluid dynamics.
\end{abstract}

\vskip.5truecm
\noindent\thanks{\em Keywords\/}:
Euler equations, analyticity

\noindent\thanks{\em Mathematics Subject Classification\/}:
35Q31, 
35A20,  
76B03  	

\section{Introduction}
In this paper, we address the existence and uniqueness of solutions of the Euler equations with initial data that are analytic near the boundary and possess Sobolev regularity away from it.

The subject of the construction of analytic solutions of the Euler equations is a classical one, with the first significant results on the topic dating back to the 1970s when Baouendi and Goulaouic proved that analyticity of the initial datum leads to analytic solutions of the Euler equations \cite{BG1975, BG1976a, BG1976b}.  They obtained their results in a general framework of analytic pseudodifferential operators theory and with a Cauchy-Kowalevskaya iteration procedure, limiting their analysis to the case when the dynamics occur on a torus or a sphere.  In such generality, the approach had limitations as it is difficult to obtain the lower bound on the radius of analyticity.  In this context, we also mention~\cite{Del1985}, where Delort extended the results of Baouendi and Goulaouic to the case of a fluid confined in a bounded analytic domain in~$\mathbb{R}^n$.

In a series of papers~\cite{BBZ1976, BB1977, Ben1979}, Bardos, Benachour, and Zerner adopted a concrete approach by writing the Euler equations in the vorticity formulation and constructed analytic solutions of the Euler equations using an iterative scheme introduced in~\cite{BF1976} in the H\"older regularity setting.  Besides giving bounds on the decay of the radius of analyticity, they also proved the persistence of the analytic regularity.  Thus, in 2D, the solution is analytic globally in time, while in 3D it is analytic as long as it remains smooth.

In~\cite{LO1997}, Levermore and Oliver analyzed a generalized Euler 2D system modeling an inviscid fluid moving in a shallow basin with varying bottom topography (for a dissipative version of the same model, see~\cite{LS2001}) and, using energy estimates in the Fourier space, they proved the persistence of Gevrey regularity of solutions.  As a byproduct, they gave an explicit estimate of the decay of the analyticity radius in terms of the $H^r$ Sobolev norm of the vorticity,~$r>5/2$. However, in terms of time dependence and size of initial data, the provided lower bound for the rate was much faster than~\cite{BBZ1976, BB1977, Ben1979}.  Kukavica and Vicol later overcame this drawback of the energy-Fourier approach~\cite{KV2009}.  In the $3D$ case, they also provided a lower bound on the linear rate of decay of the radius of analyticity depending only algebraically on the $H^r$ Sobolev norm of the vorticity.  Using the Fourier representation of the solution, these Gevrey-type results were obtained on domains without boundaries, like $\mathbb{T}^n$ or~$\mathbb{R}^n$.  In~\cite{KV2011a, KV2011b}, the same authors extended the above results to domains with boundaries.

We conclude this brief account of the previous results concerning analytic solutions of the Euler equations by mentioning the results on the propagation of the local analyticity~\cite{AM1984,AM1986,LeB1986}.  While in the classical papers~\cite{AM1984}, Alinhac and M\'etivier discussed the propagation of analyticity for hyperbolic type equations, they provided in~\cite{AM1986} the Lagrangian interior analyticity for solutions of the Euler equations. Subsequently, in~\cite{LeB1986}, Le~Bail established the propagation of Lagrangian analyticity up to the boundary. Finally, the paper~\cite{CK18} contains the preservation of the Lagrangian Gevrey radius.

The motivation for our study is threefold.  The first motivation is a recent breakthrough~\cite{LH2014,LH2019} concerning the possible blow-up of the $3D$ Euler solutions. The results have highlighted the crucial role that boundaries play in a singularity formation.  Moreover, the same results have renewed interest in the analyticity strip method.  In this approach, the spatial coordinates are regarded as complex variables, and the width of the analyticity strip, $\delta(t)$, is identified as the distance from the real axis of the complex singularity of the velocity field closest to the real space.  The concept involves observing the function $\delta(t)$ as it changes over time~$t$ and relies on the idea that a singularity in real space within the context of the Euler equations that occurs at a specific time $T$ does not come ``out of the blue''~\cite{FMB2003}, but is foreshadowed by a non-zero value of $\delta(t)$ that becomes zero at the singularity time; see~\cite{BuB2012,KSP2022, MLB2015} for some instances where the analyticity strip method has been used to monitor the complex 3$D$ Euler singularity before the possible real blow-up.  In this context, we believe that a rigorous result concerning initial data that are analytic close to the boundary only and establishing bounds on the speed at which a complex singularity can travel toward real space can be of interest.

Our second motivation derives from considering the Euler equations as the formal zero-viscosity limit of the Navier-Stokes (NS) equations.  Proving that, in a sense to be specified, the NS solutions converge to the Euler solution is a fundamental problem of the mathematical theory of fluid dynamics.  The analytic setting seems to be the most appropriate to achieve this result.  It is impossible here to review the whole literature concerning this topic.  Here, we mention~\cite{SC1998b}, where the authors constructed the NS solution as a composite asymptotic expansion involving the Euler solution and the Prandtl solution, plus an error term that goes to zero with the square root of the viscosity, thus proving that Euler and Prandtl equations, in their respective domain of validity, are good approximations of the NS equations. We also point out~\cite{NN2018,WWZ2017}, where the use of energy methods allowed the authors to prove, still for analytic data, that the Euler solution is the zero-viscosity limit of the NS solutions.

Recent advancements have shown the need to consider analyticity only near the boundary.  This noteworthy outcome was initially accomplished by Maekawa in~\cite{Mae2014}, wherein the assumption of zero vorticity close to the boundary played a crucial role.  This discovery was later extended to encompass the half-space in~\cite{FTZ2018}.  In the half-plane, Kukavica, Vicol, and Wang~\cite{KVW2020} demonstrated that for the data analytic only in the vicinity of the boundary, coupled with Sobolev regularity elsewhere, leads to the validity of the inviscid limit.  These findings were derived by employing energy-based methodologies applied to the vorticity representation of the Navier-Stokes equations without invoking the matched asymptotic expansion of the solution.  This achievement has also been expanded to a three-dimensional context, as reported in~\cite{Wan2020,KWZ}.

We also mention that the behavior of the analyticity strip of the solutions of the NS solutions has been exploited for a numerical comparison between the high-Reynolds-number solutions and the solutions of the Euler-Prandtl equations, \cite{GSS11, GSSC14a}.  This study has revealed that the structure of the NS complex singularities is much richer than what is displayed by the Euler-Prandtl solutions so a complete understanding of separation and transition to turbulence requires a deeper understanding of the analyticity strip behavior for the NS, Euler, and Prandtl equations.

Our third motivation concerns the region of analyticity of the solution to the Euler equations.  In this paper, we deduce that if the initial data for the Euler equations is analytic close to the boundary, so is the solution on a certain time interval (see~\cite{KNVW} for a detailed argument). However, the approach in~\cite{KVW2020} requires, for the normal variable, that the maximal quotient $q$ between the imaginary and real parts is less than a sufficiently small universal constant.  In fact, the parabolicity requires $q$ to be less than or equal to $1$ as the Green's function grows exponentially in the directions with a larger quotient.  Therefore, even if one could improve the approach in~\cite{KVW2020}, the limitation of the inviscid limit method gives that even if the quotient $q(0)$ is arbitrary, for the solution, one would only obtain $q$ less than 1.  As the Euler equations do not involve Navier-Stokes Green's function, the question is whether one can assume the quotient to be initially arbitrarily large and expect the solution to preserve the quotient for a positive time.  However, the same limitation on the quotient applies to Euler solutions if one uses the vorticity formulation, as in, e.g., \cite{BB1977}.  Expressing the velocity in terms of the vorticity requires using the Biot-Savart kernel, which, after the complexification of the variable, does not allow a quotient larger than 1.

Using the velocity formulation, and employing the Leray projection, we prove in this paper that the answer to the above question is affirmative.  Namely, there is no restriction on the quotient of the radii and the functional space where the solutions belong, allows for that property to continue on a positive time interval. An important ingredient in our treatment is also a nonstandard mollification procedure that we adopt in Section~\ref{subsec::regularization}; the Gaussian, the natural candidate for the mollifier in the complex setting, displays an unbounded growth along the rays with an angle larger than $\pi/4$, which would again lead to the parabolic restriction on the quotient, and is for this reason not suitable for our purposes.

In this paper, we establish the diamond-analyticity of solutions by proving that the solutions remain analytic in a complex domain shaped like a diamond, see Fig.~\ref{fig1.a}.  The diamond-type analyticity is characterized by the property that as we approach the boundary of the domain, the lower bound on the normal analyticity radius of the solution decreases linearly, and as we go far away from the boundary, the analyticity radius decreases linearly.  In contrast, the previous results in~\cite{LeB1986} on local analyticity rely on a uniform-type condition.  Observe that while the uniformly analytic data can be extended across the boundary, thus simplifying the proof, this is not the case with the diamond-analyticity.  Also, we emphasize that the behavior of the analyticity radius we have in our setting is appropriate in the study of the inviscid limit problem, where parabolicity would not allow uniform analyticity of the solution.  Finally, we note that \cite[Section~5]{KNVW} shows how to establish the uniform analyticity from the wedge analyticity if the initial data are uniformly analytic.

We now provide a concise overview of the key concepts presented in this paper.  We shall consider the Euler equations in the velocity formulation.  We suppose that the data have analytic-Sobolev regularity in the sense that, for the normal analyticity, the normal analyticity radius decays close to the boundary at a linear rate.  On the other hand, the tangential analyticity is assumed uniformly close to the boundary.  Away from the boundary, we shall assume that the data have a sufficiently high Sobolev-type regularity.  We shall work with a norm that is the sum of the analytic norm (considering the values close to the boundary) plus a Sobolev norm.  To give {\em a~priori} bounds on the analytic part of the norm, we project the equation onto the divergence-free vector field subspace through the Leray half-space projection operator~$\mathbb{P}$.  We write $\mathbb{P}$ using the Fourier variable in the tangential direction and the physical variable in the normal direction; see \eqref{Pplus_x} and~\eqref{Pplus_y}.  A~key element of our treatment is the analytic estimate of $\mathbb{P}\bu$, presented in Section~\ref{analytic_estimate}.  There, we show that the projection of a vector field enjoying the analytic-Sobolev property is still analytic-Sobolev, which is a challenge because the projection operator is nonlocal, and the value of $\mathbb{P}\bu $ close to the boundary involves values of $\bu$ away from of the boundary where it has only Sobolev-regularity.

The central two ingredients in our treatment are the usual energy estimates to bound the Sobolev part of the norm and the Cauchy estimate to bound the analytic part. However, when constructing our solution, a difficulty arises because the iteration scheme one typically uses in the analytic setting to establish convergence of solutions (see all the available proofs of the abstract Cauchy-Kowalevskaya theorem, see, e.g., \cite{Nir72, Asa88, Saf1995, LCS04}) is fully explicit,
\begin{equation}
\pardt \bu^{n+1}+\bu^n\cdot\bnabla \bu^n+\bnabla p^n=0
\fcomma
\llabel{EQ134}
\end{equation} 
while the appropriate iteration scheme for convergence in Sobolev spaces is implicit,
\begin{equation}
\pardt \bu^{n+1}+\bu^n\cdot\bnabla \bu^{n+1}+\bnabla p^{n+1}=0
\fcomma
\llabel{EQ135}
\end{equation}
since we need to take advantage of the Sobolev derivative reduction.

We circumvent this difficulty using the following steps. First, we consider the linearized version of the Euler equations~\eqref{transport.1}--\eqref{transport.3}.  Second, we perform an analytic regularization of the initial datum and the equation, which admits a globally analytic solution in space for a short time but is independent of the regularization.  Third, we pass to the limit and get an analytic-Sobolev solution of the linearized Euler equation for a time that is short but dependent only on the size of the analytic-Sobolev norm of the initial datum.  Finally, with a Cauchy-Kowalevskaya type argument and a weighted in-time analytic norm, we prove that the iteration procedure supporting energy estimates also preserves the analytic norm.  We emphasize that our iteration scheme is non-typical in the Cauchy-Kowalevskaya setting because the next-step iterate is implicitly defined; see, e.g.,~the scheme~\eqref{EQ41}.  The iterative procedure we adopt uses ideas from the paper of Asano; see~\cite{Asa88}.

\section{The set up and the main result}
\colb
We write the Euler equations in $\mathbb{R}^2_+$ as
  \begin{align}
   &
  \pardt \bu + \bu\cdot \bnabla\bu+\bnabla p =0
  \fcomma
  \label{EQ12}
  \\&
  \bnabla\cdot \bu=0
  \label{EQ13} \text{~and~} \ u_2|_{y=0}=0
  \fcomma
  \\&
  \bu|_{t=0}=\buin
  \fcomma
  \label{EQ14}
  \end{align}
where $\bu=(u_1(x,y,t),u_2(x,y,t))$, with $x\in \mathbb R$ (or $x\in {\mathbb {R}}^2$) and $y\in \mathbb{R}^+$, is the velocity field.

To simplify the presentation, we restrict ourselves to two space dimensions; all theorems and proofs also cover the space dimension three by considering the variable $x$ as a vector rather than a scalar ($x\in \mathbb{R}$ replaced by $x\in\mathbb{R}^{2}$). Higher space dimensions can also be covered by increasing the Sobolev exponents.

Introducing the Leray projector for the half space, $\mathbb{P}$, we may write the above equations
equivalently as
\be
\bu=\mathcal{F}(\bu,t)
\fcomma
\qquad \mbox{where} \quad \mathcal{F}(\bu,t) \equiv
\buin-\int_0^t \mathbb{P}\left(\bu\cdot\bnabla\bu\right)\,ds    \label{euler.operator}
\fperiod
\ee

The explicit expression of the projection operator is given in \eqref{Pplus_x} and \eqref{Pplus_y} below, where we use a mixed representation, Fourier in the $x$-variable 
and physical in the $y$-variable. 
This representation is the most appropriate to get estimates on the
projection operator in the functional setting we shall adopt; see the next section. 
Our main result is the following.

\cole
\begin{theorem}
\label{T01}
Let $0<\theta_0<\pi/2$ and $m\geq 3$. Suppose that $\buin\in H^m_{\theta,\text{D}}\cap H^m_{\text a}$. Then there exists $\beta>0$ 
such that the system \eqref{EQ12}--\eqref{EQ14} admits a unique solution $\bu\in H^m_{\theta-\beta t,\text{D}}\cap H^m_a $ for $t\in [0,\theta/2\beta[$.
\end{theorem}
\colb

Above and in the rest of the paper, if $\bu$ is a vector function, then by $\bu\in H^m_{\theta, \text{D}}$ we mean that all the components of $\bu$ belong to~$ H^m_{\theta, \text{D}}$. The same notational convention holds for other functional spaces.

The definition of the function spaces $H^M_{\theta,\text{D}}$ and $H^m_a$ is given below. Roughly speaking,  $H^M_{\theta,\text{D}}$ is the space of 
functions $f(x,y)$  analytic with respect to $y$ in the region $D_\theta$;
see  Fig.~\ref{fig1.a},   
and with  $x$-Fourier spectrum  which exponentially decays at a rate that 
degrades to zero at the distance $1+\theta$ from the boundary.

\section{Function spaces}

\subsection{Analyticity close to the boundary: analytic norms and  function spaces}
We define the domain of analyticity in the $y$-variable as a diamond-shaped open set
$D_\theta$, see Fig.~\ref{fig1.a}, defined by
\begin{figure}
	\centering
	\begin{subfigure}{0.45\textwidth}
		\includegraphics[alt={analytic domain in $y$},width=\textwidth]{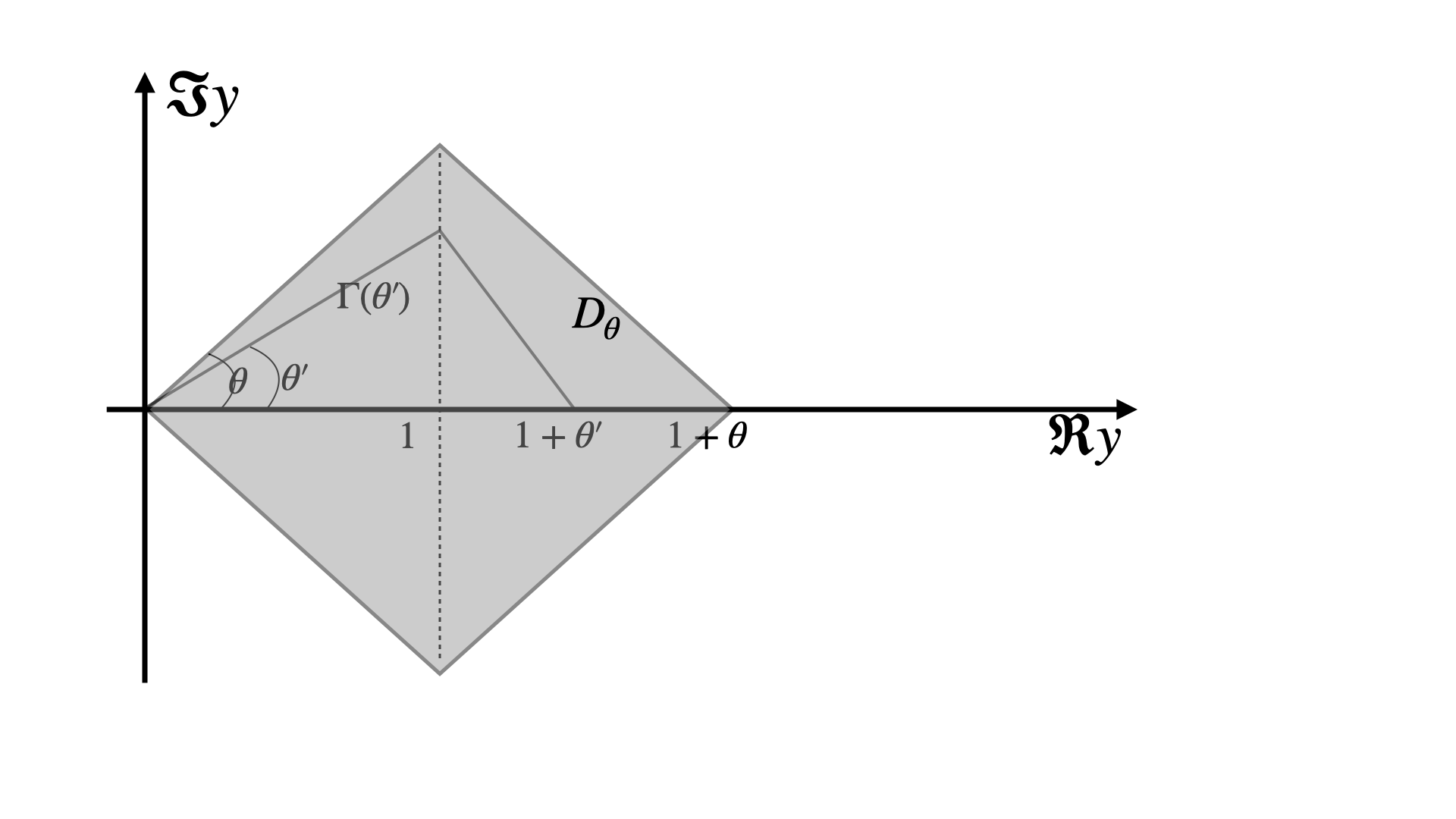}
		\caption{The domain of analyticity in the $y$ variable. The angle $\theta$ satisfies $0<\theta<\pi/2$. }
		\label{fig1.a}
	\end{subfigure}
	\hfill
	\begin{subfigure}{0.45\textwidth}
		\includegraphics[alt={analytic domain in $x$},width=\textwidth]{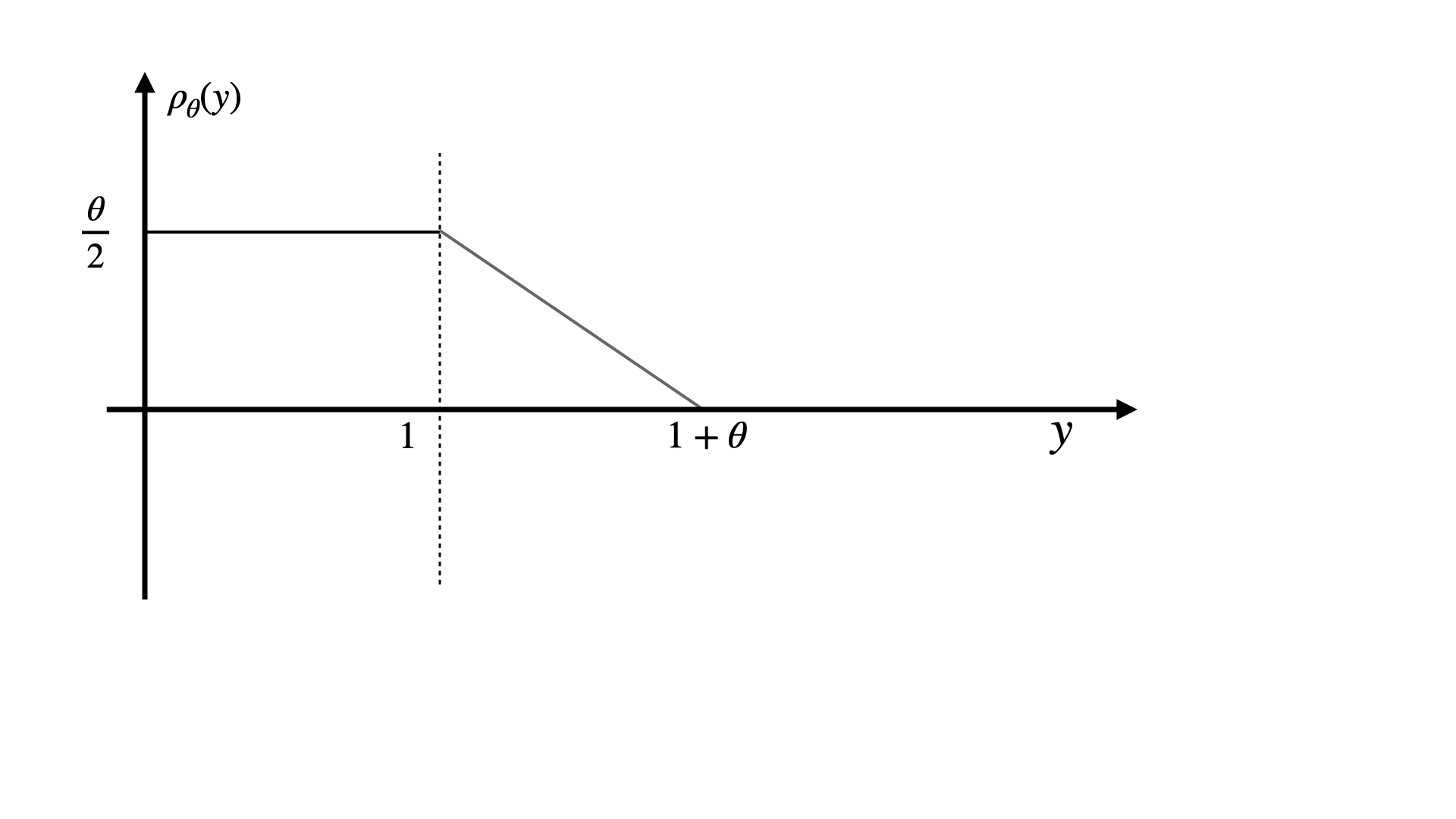}
		\caption{The rate of exponential decay of the spectrum as a function of the distance from the boundary.}
		\label{fig1.b}
	\end{subfigure}
	\hfill
	\caption{Domain of analyticity for functions analytic close to the boundary.}
	\label{fig:figures}
	\end{figure}
  \begin{align}
    \begin{split}
   D_\theta&=\Bigl\{ z\in \mathbb{C}: \; 0<\Re{z}\leq 1, \; |\Im{z}| < \Re{z}\tan{\theta} \Bigr\}
   \\&\indeq
   \bigcup 
   \left\{ z\in \mathbb{C}: \; 1\leq\Re{z}< 1+\theta, \; |\Im{z}|< \left(1+\theta-\Re{z}\right)\frac{\tan{\theta}}{ \theta}\right\}
   \fperiod
  \end{split}
   \label{EQ85}
   \end{align}
We next introduce the $L^2$-based norm for functions~$f$ depending on $x\in \mathbb{R}$ and $y\in D_\theta$.
First, we define the path of integration in the $y$-variable,
  \begin{equation}
    \begin{split}
   \Gamma(\theta')&=\left\{ z\in \mathbb{C}: \; \Re{z}\leq 1, \; \Im{z}= \Re{z}\cdot\tan{\theta'} \right\}
    \\&\indeq
    \bigcup 
   \left\{ z\in \mathbb{C}: \; 1\leq\Re{z}\leq 1+\theta', \; \Im{z}= \left(1+\theta'-\Re{z}\right)\frac{\tan{\theta'}}{ \theta'}\right\}
  \fperiod
  \end{split}
  \label{EQ86}
  \end{equation}
Also, we introduce a $y$-dependent function $\rho_\theta\colon [0,1+\theta]\rightarrow \mathbb{R}$ as 
\begin{equation}
\rho_\theta(s)=
    \begin{cases}
       \theta/2 & 0\leq s\leq 1 \\
                (1+\theta-s)/2 & 1\leq s \leq 1+\theta
\fsemicolon
    \end{cases}
   \llabel{EQ88}
   \end{equation}
see Fig.~\ref{fig1.b},
expressing the rate of exponential decay of the Fourier spectrum in the $x$-variable. 
Then, the diamond norm, or simply the D-norm, of $f(x,y)$ is defined as
\be
|f|^\text{D}_{\theta}=  \left(  \sup_{0\leq \theta'\leq \theta} \int_{\Gamma(\theta')} | dy|  \int_\mathbb{R}d\xi  e^{2\rho_\theta(\Re{y})|\xi|} |\hat{f}(\xi,y)|^2  \right)^{1/2}
\label{EQ101}
\fperiod
\ee
In the above definition, we have denoted by $\hat{f}$ the Fourier transform in the $x$-variable,
  \begin{equation}
   \hat f(\xi)=\int_{-\infty}^\infty f(x)e^{-ix\xi} dx
   \fperiod 
   \label{EQ59}
   \end{equation}
In the sequel, unless necessary for clarity, we shall not distinguish in notation between a function and its Fourier transform. 
We now finally define the analytic space of $(x,y)$-dependent functions involving derivatives up to order~$m$.

\begin{definition}
Let $\theta>0$ and $m\in\mathbb{N}_0$.
The space $H^{m}_{\theta,\text{D}}$ is a set of the $(x,y)$-dependent functions such that
\be
|f|^\text{D}_{m,\theta}=\sum_{i+j\leq m} |\pardx^i\pardy^j f|_{\theta}
<\infty
\fperiod
\label{EQ2.2}
\ee

\end{definition}

We now introduce the time-dependent norms. 
	\begin{definition}
Let $\beta,\theta>0$ and $m\in\mathbb{N}_0$. 
For a function $f$ depending on $(x,y,t)$, we say that  $f\in H^m_{\theta,\beta,D}$
when
$f\in C([0,\tau],H^{m}_{\theta-\beta \tau,\text{D}})$,
for all~$\tau$ such that  $ 0<\tau<\theta/\beta$.
For such $f$, we set
  \begin{equation}
		|f|^\text{D}_{m,\theta,\beta}=\sup_{\theta-\beta t >0}|f(\cdot,\cdot,t)|^\text{D}_{m,\theta-\beta t}
	\fperiod
	\label{EQ119}\end{equation} 
  \end{definition}

\subsection{Analyticity on the half-space: analytic norms and  function spaces}

We now define norms for functions that are analytic in $y$ for all $0<\Re{y}<\infty$ and whose $x$-Fourier spectrum is exponentially decaying uniformly in~$\Re{y}$. 
Introducing 
the conoid
  \begin{equation}
    \begin{split}
  \text{C}_\theta
    &=\bigl\{ z\in \mathbb{C}: \; 0<\Re{z}\leq 1, \; |\Im{z}| < \Re{z}\tan{\theta} \bigr\}
   \\&\indeq
   \bigcup 
  \bigl\{ z\in \mathbb{C}: \; 1\leq\Re{z}< \infty, \; |\Im{z}|< \tan{\theta} \bigr\}
  \end{split}
   \label{EQ47}
   \end{equation}
   \begin{figure}
   	\centering
   		\includegraphics[alt={conoid},scale=0.15]{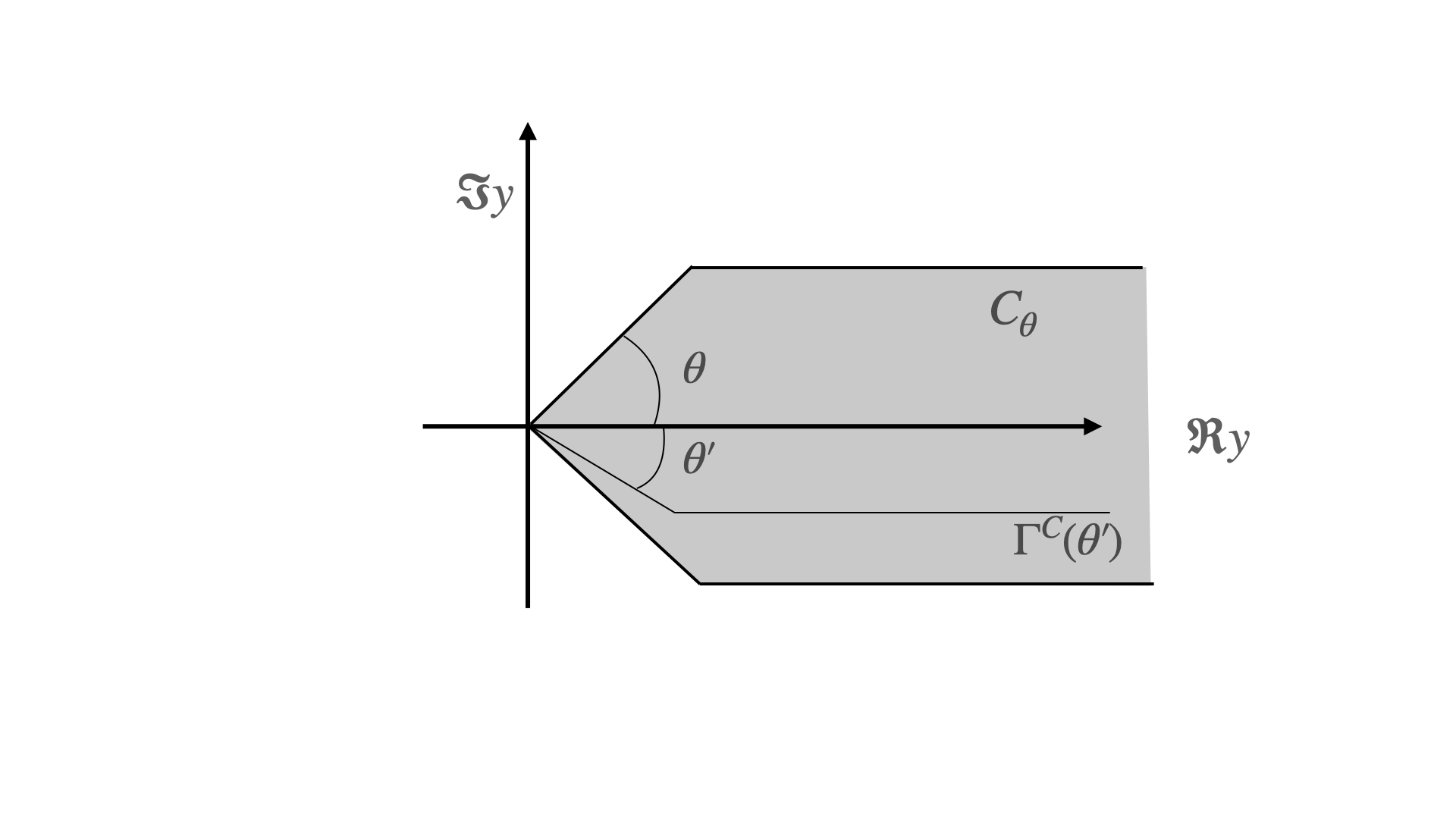}
   		\caption{Domain of analyticity for functions analytic on the half-plane.}
   		\label{fig2}
   	\end{figure}
(see Fig.~\ref{fig2})
and the path of integration in the complex plane 
  \begin{equation}
    \begin{split}
  \Gamma^\text{C}(\theta')&=\left\{ z\in \mathbb{C}: \; \Re{z}\leq 1, \; \Im{z}= \Re{z}\cdot\tan{\theta'} \right\}
   \\&\indeq
   \bigcup 
  \left\{ z\in \mathbb{C}: \; 1\leq\Re{z}< \infty, \; \Im{z}=
  \tan{\theta'}\right\}
   \comma   |\theta'|<\theta
  \fcomma
  \end{split}
  \label{EQ48}
   \end{equation}
with the rate of exponential decay of the $x$-Fourier spectrum
  \begin{equation}
   \rho^\text{C}_\theta(s)=\frac{\theta}{2}
   \fcomma
   \label{EQ89}
   \end{equation}
we define the conoid norm, or simply the C-norm, of $f$ as
\be
   |f|^\text{C}_{\theta}=  \left(\sup_{0\leq \theta'\leq \theta} \int_{\Gamma^\text{C}(\theta')}  |dy|  \int_\mathbb{R}d\xi  e^{2\rho^\text{C}_\theta|\xi|} |\hat{f}(\xi,y)|^2 \right)^{1/2}
   \label{analytic.norm.C}
   \fcomma
  \ee
which is an analog of the norm defined in \eqref{EQ101} for functions analytic close to the boundary.

\begin{definition}
Let $\theta>0$ and $m\in\mathbb{N}_0$.
The space $H^{m}_{\theta,\text{C}}$ is a set of the $(x,y)$-dependent functions such that
  \begin{equation}
  |f|^\text{C}_{m,\theta}=\sum_{i+j\leq m} |\pardx^i\pardy^jf|^\text{C}_{\theta}
   <\infty
   \fperiod
   \label{EQ90}
   \end{equation}
\end{definition}
Note that the norm \eqref{EQ90} is an analog of~\eqref{EQ2.2}
for functions analytic close to the boundary.

\begin{definition}
Let $\theta,\beta>0$ and $m\in\mathbb{N}_0$.
The space $H^{m}_{\theta,\beta,\text{C}}$ is a set of the $(x,y,t)$-dependent functions such that
\be
|f|^\text{C}_{m,\theta,\beta}
=
\sup_{\theta-\beta t >0}
|f(t)|^\text{C}_{m,\theta}   
=
\sup_{\theta-\beta t >0}
\sum_{i+j\leq m} |\pardx^i\pardy^j f|^\text{C}_{\theta}<\infty      
\fperiod
\label{EQ90t}
\ee
\end{definition}
Note that the norms \eqref{EQ90} and \eqref{EQ90t} are analogs of~\eqref{EQ2.2} and~\eqref{EQ119} 
for functions analytic close to the boundary.

\subsection{Sobolev norms}
Let~$a>0$ and $m\in\mathbb{N}_0$. For an $(x,y)$-dependent function $f$, we introduce the notation
\begin{equation}
\|f\|_a=\left(\int_a^{\infty}\int_\mathbb{R}|f(y)|^2\,dx dy\right)^{1/2}
   \label{EQ49}
   \end{equation}
and the Sobolev norms 
  \begin{equation}
   \|f\|_{m,a}=\sum_{i+j\leq m} \| \pardx^i\pardy^jf\|_a
   \fperiod
   \label{EQ50}
  \end{equation}
Also, the space
$H^m_a$ is the set of $(x,y)$-dependent functions such that
  \begin{equation}
   \|f\|_{m,a}<\infty
   \fperiod
   \label{EQ09}
\end{equation}

In the rest of the paper, we fix $a=1/2$ (even though any other value
$a\in(0,1]$ would work), but we keep indicating it in the notation for
the norms.
The value of $a$ is chosen so that the set
$\{y=a\}$ is well inside the analyticity region. In the energy estimates, this allows us to estimate the boundary terms arising at $y=a$ using analyticity. 

\subsection{The weighted norm}
For a fixed $0<\gamma<1$, we introduce a weighted norm 
\be
| f|^{(\gamma)}_{m,\theta,\beta}=\sup_{0\leq \beta t\leq \theta-\theta'} \left(1-\frac{\beta  t}{\theta-\theta'}\right)^\gamma |f(\cdot,\cdot,t)|^\text{D}_{m,\theta'}
   \fperiod
   \label{EQ121}
\ee
It is easy to check that
\be
| f|^{(\gamma)}_{m,\theta,\beta'} \leq 	|f|^\text{D}_{m,\theta,\beta'}
   \leq
   \left(1-\frac{\beta}{\beta'}\right)^{-\gamma} |f|^{(\gamma)}_{m,\theta,\beta}
    \comma \beta'>\beta
    \fperiod
\label{Asano_estimate}
\ee
The inequalities in~\eqref{Asano_estimate} are
crucial in proving the convergence of a sequence of approximations to the solution of the Euler equations.

\subsection{The combined  norms}
We now introduce the norms
  \begin{equation}
	\opnorm{f}_{m,\theta,a}
	=
	|f|^{\text{D}}_{m,\theta}
	+ \norm{f}_{m,a}
	\fcomma
	\label{EQ02}
\end{equation}
\begin{equation}
|||f|||_{m,\theta,\beta,a} =|f|^\text{D}_{m,\theta,\beta}+\|f\|_{m,a}
	\fcomma
\label{EQ123}\end{equation}
with
\begin{equation}
|||f|||^{(\gamma)}_{m,\theta,\beta,a} =|f|^{(\gamma)}_{m,\theta,\beta}+\sup_{0\leq t<\theta/\beta}\|f(\cdot,\cdot,t)\|_{m,a}
\fcomma
\label{EQ124}\end{equation}
which combines  the analytic  and Sobolev norms.

\subsection{Norms with a stopping time}
We introduce the norms with a stopping time~$T$ because in the energy estimates
we need to have a finite width of the strip of analyticity in order to bound the terms that 
arise at $y=a$ using the Cauchy estimate.
 The weighted norm with a stopping time~$T$ reads 
\be
| f|^{(\gamma)}_{m,\theta,\beta,T}=\sup_{\substack{0\leq \beta t\leq \theta-\theta'\\0\leq t\leq T}} \left(1-\frac{\beta  t}{\theta-\theta'}\right)^\gamma 
|f(\cdot,\cdot,t)|^\text{D}_{m,\theta'}
\comma 0<T<\frac{\theta}{\beta}
\fcomma
\label{EQ122}
\ee
while the norm with the stopping time is defined as
\begin{equation}
|||f|||^{(\gamma)}_{m,\theta,\beta,a,T} =|f|^{(\gamma)}_{m,\theta,\beta,T}+\sup_{0\leq t<T}\|f(\cdot,\cdot,t)\|_{m,a} \qquad \mbox{with} \quad 0<T<\theta/\beta
\fperiod
\label{EQ.comb.stop}\end{equation}
Immediate consequences of the above definitions and of the estimate \eqref{Asano_estimate} are the inequalities
	\begin{align}
	||| f|||^{(\gamma)}_{m,\theta,\beta'} &\leq 	|||f|||^\text{D}_{m,\theta,\beta'} \leq  (1-\beta/\beta')^{-\gamma} |||f|||^{(\gamma)}_{m,\theta,\beta} 
	\label{Asano_estimate_I}
	\end{align}
and
	\begin{align}
	||| f|||^{(\gamma)}_{m,\theta,\beta',T} &\leq 	|||f|||^\text{D}_{m,\theta,\beta',T} \leq  (1-\beta/\beta')^{-\gamma} |||f|||^{(\gamma)}_{m,\theta,\beta,T}  
	\fcomma
	\label{Asano_estimate_I||} 
	\end{align}
where $\beta'>\beta$.

\section{The analytic estimates} 
\label{analytic_estimate}
\subsection{The projection operator}

This section contains the crucial estimates on the interplay between the D-norm and the Sobolev norms with the projection operators and the bilinear form.

The tangential and normal components, respectively,  of the half-space projection operator $\mathbb{P}$ have the expressions
  \bea
   \mathbb{P}_\tau \bu =u -\frac{1}{2} |\xi|
   \Biggl( \int_0^y \,dy' e^{-|\xi|(y-y')} (u+Rv) +
   \int_0^y \,dy' e^{-|\xi|(y+y')} (u-Rv) \nonumber \\
   +
   \left( 1 +e^{-2|\xi|y}\right)\int_y^\infty \,dy'   e^{ |\xi|(y-y')}
   (u-Rv)
   \Biggr)
   \label{Pplus_x}
  \eea
and
  \bea
   \mathbb{P}_n \bu =\frac{1}{2} |\xi|
   \Biggl(\int_0^y \,dy' e^{-|\xi|(y-y')} (-Ru+v)
   -\int_0^y \,dy' e^{-|\xi|(y+y')}(Ru+v)   \nonumber \\
   +\left( 1 -e^{-2|\xi|y}\right) \int_y^\infty \,dy' e^{ |\xi|(y-y')}
   (Ru+v) \Biggr)
  \fcomma
  \label{Pplus_y}
  \eea
where 
  \begin{equation}
   R=\frac{i\xi}{|\xi|}
   \llabel{EQ36}
   \end{equation}
is the Riesz transform with respect to the tangential variable~\cite{SC1998a}.
As usual, we denote by $\xi$ the Fourier variable of~$x$, omitting to distinguish between a function $f(x)$ and its Fourier transform $\hat{f}(\xi)$ defined in \eqref{EQ59}.
Using the above  expressions,
we have the following estimate for~$\mathbb{P}$.

\cole
\begin{proposition} \label{prop_projection}
Let $\bu\in H^m_{\theta,\text{D}}\cap H^m_a$.
Then $\mathbb{P}\bu \in H^m_{\theta,\text{D}}$, and we have 
  \begin{equation}
   |\mathbb{P}\bu|^\text{D}_{m,\theta}
   \leq
   \cpro
   \opnorm{u}_{m,\theta,a} 
   \fperiod
   \llabel{EQ52}
  \end{equation}
\end{proposition}
\colb

Here and in the sequel, we indicate by $\cpro$ a constant resulting from the estimate of the projection operator. 
Proposition~\ref{prop_projection} is proven
in Section~\ref{subsect_projection_est}.
The next proposition shows an interesting property of the projection operator. 
Although $\mathbb{P}$  involves (by means of integration in the
$y$-variable) the domain of non-analyticity of a function~$\bu$, it nevertheless supports the
Cauchy estimate.

\cole
\begin{proposition}
\label{P01}
Let $m\geq3$ and 
$\bu,\bv\in H^m_{\theta',\text{D}}\cap H^m_a$ with $\bu$ and $\bv$ divergence-free
and $\gamma_n \bu^2 = \gamma_n \bv^2 =0$.
Suppose that $\theta<\theta'$. Then
$\mathbb{P}(\bv\cdot\bnabla\bu)\in H^m_{\theta,\text{D}}$, and we have an estimate
  \begin{equation}
   |\mathbb{P}(\bv\cdot\bnabla\bu)|^\text{D}_{m,\theta}
     \leq
     \frac{\cpCau}{\theta'-\theta}
     \bigl(|\bv|^\text{D}_{m,{\theta'}}+ \|\bv\|_{m,a}\bigr)
     \bigl(|\bu|^\text{D}_{m,\theta'}+ \|\bu\|_{m,a}\bigr)
    \fperiod
   \label{EQ11}
  \end{equation}
\end{proposition}
\colb

Here and in the sequel, we use $\cpCau$ to indicate a constant deriving from the estimate of the projection operator and from the use of the Cauchy estimate. 
The next lemma is a consequence of the Cauchy estimate.

\cole
\begin{lemma} \label{P02}
Suppose that $\buo$ and $\but$ belong to $H^{m}_{ \theta',\text{D}}$,
where
	$m\geq 3$, 
	with $\gamma_n \buo = \gamma_n \but =0$.
For $0<\theta<\theta'$, we have
	\be
	|  \buo \cdot \bnabla \buo -\but \cdot \bnabla \but |^\text{D}_{m, \theta}
	\leq
	\cCau
	\frac{(|\buo|^\text{D}_{m,\theta'} + | \but|^\text{D}_{m , \theta'})|\buo - \but|^\text{D}_{m , \theta'}}{\theta' - \theta}
	\llabel{EQ82old}
	\fperiod
	\ee
\eold
\end{lemma}
\colb

By $\cCau$, we indicate a constant derived from  the Cauchy estimate for analytic functions.

The next proposition is crucial in our treatment of the Euler equations in the D-analytic space.

\cole
\begin{proposition}
\label{P04}
Let $m\geq3$  and $\bu^{1},\bu^{2}\in H^m_{\theta',\text{D}}\cap H^m_a$,
where $a\in[0,1]$,
with $\bu^{1}$ and $\bu^{2}$ divergence-free
and $\gamma_n \bu^1_2 = \gamma_n \bu^2_2 =0$.
For $0<\theta<\theta'$, we have
  \begin{equation}
    \begin{split}
   &
   | \mathbb{P}(\bu^{1}\cdot\bnabla\bu^{1})
    -\mathbb{P}(\bu^{2}\cdot\bnabla\bu^{2})
   |^\text{D}_{m,\theta}
    \\&\indeq
     \leq
     \frac{\cpCau}{\theta'-\theta}
     \left(
        |\bu^{1}|^\text{D}_{m,\theta',a}+ \Vert\bu^{2}\Vert_{m,a}\right)
      \left(|\bu^{1}-\bu^{2}|^\text{D}_{m,\theta',a} + \Vert\bu^{1}-\bu^{2}\Vert_{m,a}\right)
    \fperiod
  \end{split}
  \label{EQ116}
  \end{equation}
\end{proposition}
\colb

From this proposition, 
it follows that the operator $\mathcal{F}$, defined in \eqref{euler.operator}, is quasi-contractive.
This, ultimately, will lead us to prove 
that the iteration scheme we shall introduce in Section~\ref{sec06} is contractive in the appropriate time-weighted function space;
see Section~\ref{subsec:contractiveness} and Proposition~\ref{prop:contractiveness} below.

Proposition~\ref{P04} follows immediately from
Proposition~\ref{P01} and Lemma~\ref{P02}.

\section{The linearized problem}

In this section, we analyze
the linearized Euler
equation
  \begin{align}
   &\pardt \bu +\bv\cdot \bnabla \bu +\bnabla p=0
   \fcomma
   \label{transport.1}
   \\&
   \bnabla\cdot\bu=0
   \fcomma
   \label{transport.2}
   \\&
   \bu|_{t=0}=\buin \label{transport.3}
   \fcomma
  \end{align}
for $t\in[0,\Tf]$.
\colb
In the next section, $\bv$ and $\bu$ represent consecutive iterates
of the approximation scheme.
The main hypothesis in this section is
$\bv\in H^m_{\theta_0-\beta t,\text{D}}\cap H^m_a$ and
$\buin\in H^m_{\theta_0,\text{D}} \cap H^m_a$. 

The result that we shall prove is the following.

\cole
\begin{theorem}\label{T.linear}
Suppose that $\bv\in H^m_{\theta_0-\beta t,\text{D}}\cap H^m_a$ and $\buin\in H^m_{\theta_0,\text{D}} \cap H^m_a$
with $\theta_0<\pi/2$.
Then, for $\beta'>\beta$ and $0<\bar{\theta}<\theta_0/2$, the system~\eqref{transport.1}--\eqref{transport.3} admits a unique solution $\bu\in H^m_{\theta_0-\beta' t,\text{D}}\cap H^m_a$ for $0<t<(\theta_0-\bar{\theta})/\beta'$. 
\end{theorem}
\colb

The proof is achieved in several steps. 
In the first step, we solve a regularized version of the system. 

\subsection{The regularized version: existence and uniqueness of a solution analytic in the conoid, with a large norm} \label{subsec::regularization}
The main result of this subsection is Proposition~\ref{exist.transport.regularized}, where we show 
that the regularized version of the Euler equation \eqref{transport.reg.1}--\eqref{transport.reg.3} has a unique solution for a time 
which is independent of the size of the regularization. 
The initial datum is analytic close to the boundary, and the regularization extends analyticity 
to the whole space so that the solution is analytic in the
conoid $C_\theta$, with $0<\theta<\pi/2$. 
Note that the regularization in the normal variable with a Gaussian would lead to the restriction  $0<\theta<\pi/4$ due to the unbounded growth 
of the Gaussian when $|y|\rightarrow \infty$ when $\Im{y}>\Re{y} $. 
For this reason we shall resort to a new mollifier.
 We introduce the  function
\begin{equation}
I(y)= A \left[ \frac{1}{(y+i)^4}+\frac{1}{(y-i)^4}\right]
   \fcomma
   \llabel{EQ57}
   \end{equation}
where $A$ is constant chosen so that 
   $$
   \int_{-\infty}^\infty I(y) dy=1\fperiod
   $$
Note that $I(y)$ is real valued when $y\in\mathbb{R}$;  moreover it has two poles on the imaginary axis, while it is analytic on any 
angular sector of the complex plane $|\Im{y}|<|\Re{y}|\tan{\theta} $ with the angle satisfying $\theta<\pi/2$.
In this angular sector, it is also bounded by
$C_{\theta}/(|y|^2+1)$.
   
We then define a regularization operator $\Jreg$, acting on a function $\vphi(x,y)$, with $(x,y)\in \mathbb{R}\times \mathbb{R}^+$, as
\begin{equation}
\Jreg \vphi
    = \frac{1}{\ep}e^{-|\xi|^2/2}I(y/\ep)*_y E \vphi
   \fcomma
   \llabel{EQ58}
   \end{equation}
where  $E\vphi$ is the Sobolev extension of $\vphi$ to~$\mathbb{R}^2$. 

The following lemma bounds the conoid norm $|\cdot |^\text{C}$ in terms of the Sobolev norm~$\|\cdot \|_{a,m}$.  

\cole
\begin{lemma}
Suppose that $\vphi\in H^m$. Then
$\Jreg \vphi \in H^m_{\theta,\text{C}}$0, and we have the estimates
  \begin{align}
  \begin{split}
   &
  |\Jreg\vphi|^{\text{C}}_{m,\theta}\leq c {e^{c{\theta}/\ep}} \|\vphi \|_{m}
  \end{split}
   \llabel{EQ125}
  \end{align}
and
  \begin{align}
  \begin{split}
    |\bnabla \Jreg\vphi|^{\text{C}}_{m,\theta}
    \leq c {e^{c{\theta}/\ep}}\|\vphi\|_{m}
  \fcomma
  \end{split}
   \llabel{EQ01}
  \end{align}
where the constants do not
depend on $\vphi$ and~$\ep$. 
\end{lemma}
\colb

The next lemma is an immediate consequence of the previous statement, given that  $ \|\vphi\|_{m}\leq |\vphi|^\text{D}_{m,\theta} +\|\vphi \|_{m,a}$.

\cole
\begin{lemma} \label{estimate.C.versus.D}
Suppose that $\vphi\in H^m_{\theta,\text{D}}\cap H^m_a$. Then $\Jreg \vphi \in H^m_{\theta,\text{C}}$, and we have the estimates
  \begin{align}
  \begin{split}
   &
   |\Jreg\vphi|^{\text{C}}_{m,\theta}\leq c {e^{c{\theta}/\ep}} \left(|\vphi|^{\text{D}}_{m,\theta} +\|\vphi \|_{m,a}\right)
\fcomma
 \\&
  |\bnabla \Jreg\vphi|^{\text{C}}_{m,\theta}\leq c {e^{c{\theta}/\ep}} \left(|\vphi|^{\text{D}}_{m,\theta} +\|\vphi\|_{m,a}\right)
 \fcomma
  \end{split}
   \label{EQ03}
  \end{align}
where $c$ is a  constant that does not depend on $\vphi$ and~$\ep$. 
\end{lemma}
\colb

In this subsection, we consider a regularized version of the equations \eqref{transport.1}--\eqref{transport.3},
  \begin{align}
   &
     \pardt \bureg + \Jreg \left( \bv\cdot \bnabla \Jreg\bureg \right)+\bnabla p_\ep=0 \label{transport.reg.1}
     \fcomma
   \\&
     \bnabla\cdot\bureg=0
    \fcomma
    \label{transport.reg.2}
   \\&
      \bureg|_{t=0}=\Jreg\buin \label{transport.reg.3}
    \fcomma
  \end{align}
which may be rewritten as
\be
\bureg=\int_0^t \Freg (\bv,\bureg,s)ds+\Jreg\buin
\label{transport.reg.operator.form}
\fcomma
\ee
where
\be
\Freg(\bv,\bu,s) \equiv -\mathbb{P}\Jreg\bnabla \cdot \left(\bv\otimes\Jreg \bu\right)
   \label{EQ126}
\fperiod
\ee
The following estimate is an immediate consequence of Lemma~\ref{estimate.C.versus.D} and the estimate on the projection 
operator in the norm $|\cdot|^{\text C}$; see~\cite{SC1998a}. 

\cole
\begin{lemma}
Suppose that $\bv\in H^m_{\theta_0, \beta,\text{D}}\cap H^m_a$ and $\bureg\in  H^m_{\theta_0,\beta,\text{C}}$.
Then
  \begin{align}
   &
   |\Freg ( \bv,\bureg,s)|^{\text C}_{m,\theta_0-\beta t}
   \leq c e^{c\theta_0/\ep} \left(  |\bv|^{\text{D}}_{m,\theta_0-\beta t} +\|\bv\|^m_a  \right)|\bureg|^{\text{C}}_{m,\theta_0-\beta t}
   \label{est.Freg.time.depen}
   \end{align}
and
  \begin{align}
   |\Freg( \bv,\bureg,s)|^{\text{C}}_{m,\theta_0,\beta}
   \leq c e^{c\theta_0/\ep} \left(  |\bv|^{\text{D}}_{m,\theta_0,\beta}+ \sup_{0\leq t\leq\theta_0/\beta}\|\bv\|^m_a  \right)|\bureg|^{\text{C}}_{m,\theta_0,\beta } 
   \fcomma
   \label{est.Freg.time.indepen}     
  \end{align}
for $t\in[0,\Tf]$,
where the constant $c$ does not depend on $\bv$, $\bureg$, or~$\ep$. 
\end{lemma}
\colb

On the other hand, the operator $\Freg$ can be bounded in the $|\cdot|^{\text{D}}$ norm uniformly in~$\ep$.
This is expressed by the following lemma, which is an immediate consequence of the estimate on the projection operator 
$\mathbb{P}$ in the $|\cdot|^{\text{D}}$ norm and the Cauchy estimate; see Proposition~\ref{P01}.

\cole
\begin{lemma} \label{est.Freg.uniform}
Suppose that
$\bv\in H^m_{\theta_0, \beta,\text{D}}\cap H^m_a$ and $\bureg\in  H^m_{\theta_0,\beta,\text{C}}$.
Then
 \be
 |\Freg( \bv,\bureg,s)|^{\text{D}}_{m,\theta'}
 \leq
 \cpCau \left(  |\bv|^{\text{D}}_{m,\theta'}+ \|\bv\|^m_a  \right) \frac{|\bureg|^{\text{D}}_{m,\theta}+ \|\bureg\|_{m,a} }{\theta-\theta'}
\fperiod
 \ee
\end{lemma}
\colb

The following proposition shows that if the initial data $\buin$ and $\bv$ have analytic regularity near the boundary and Sobolev regularity away from the boundary, then the above-regularized version of the pressure-transport equation admits a unique solution that is analytic for $y\in ]0,\infty[$.

\cole
\begin{proposition} \label{exist.transport.regularized}
Suppose that $\buin\in H^m_{\theta_0,\text{D}} \cap H^m_a$ and $\bv\in H^m_{\theta_0-\beta t,\text{D}}\cap H^m_a$. 
 Then the system \eqref{transport.reg.1}--\eqref{transport.reg.3} admits a unique solution $\bureg\in H^m_{\theta_0-\beta t,\text{C}}$. 
\end{proposition}
\colb

\begin{proof}
Define the sequence
\begin{equation}
   \bureg^{n+1}=\int_0^t \Freg(\bv, \bureg^n,s)ds +\Jreg\buin, \qquad \bureg^0=\Jreg \buin
\fperiod
   \llabel{EQ60}
   \end{equation}
Let $L\equiv c e^{c\theta_0/\ep} |||\bv|||_{m,\theta,\beta}$ be the constant appearing in the estimate of $\Freg$ given in \eqref{est.Freg.time.indepen}, and define $\alpha\equiv Lt$. 
Then, one may write
  \begin{align}
  \begin{split}
   |\bureg^{n+1}|^{\text{C}}_{m,\theta_0,\beta}
   &\leq
   |\Jreg \buin|^{\text{C}}_{m,\theta_0}+ |\bureg^{n}|^{\text{C}}_{m,\theta_0,\beta}L\int_0^t ds
   \leq 
   |\Jreg \buin|^{\text{C}}_{m,\theta_0}+\alpha |\bureg^{n}|^{\text{C}}_{m,\theta_0,\beta}
   \\&
   \leq 
   |\Jreg \buin|^{\text{C}}_{m,\theta_0}\left(1+\alpha +\alpha^2+\ldots +\alpha^{n+1}  \right)
   \leq 
   2  |\Jreg \buin|^{\text{C}}_{m,\theta_0}
   \\&
   \leq
   2 c e^{c\theta_0/\ep}\left(|\buin|^{\text{D}}_{m,\theta_0}+|\buin|_{m,a}\right)   
  \fcomma
  \end{split}
  \label{est.recursive.transport}
  \end{align}
where we have chosen $t$ to be such that  $\alpha =Lt=1/2$. 
 
Using the same arguments, one may write
 \bea 
 |\bureg^{n+1}-\bureg^{n}|^{\text{C}}_{m,\theta_0,\beta}\leq \alpha |\bureg^{n}-\bureg^{n-1}|^{\text{C}}_{m,\theta_0,\beta}\leq 
 \alpha^{n+1}  |\Jreg \buin|^{\text{C}}_{m,\theta_0} \label{est.Cauchy.transport} 
\fcomma
 \eea
which shows that  $\bureg^n$ is a Cauchy sequence. 
We then proceed as follows:
  \begin{enumerate}
   \item Observe that the constant appearing in \eqref{est.recursive.transport} does not depend on $\bv$, $\bureg$, and~$\ep$.  
   \item The  estimate \eqref{est.recursive.transport}  says that the sequence $\bureg^n$, starting from $\Jreg \buin$, remains in a ball whose radius in the norm $|\cdot |^\text{C}$
   is two times larger that the norm of~$\Jreg \buin$. 
\item The above remark
gives that $\bureg^n$ converges to a solution of \eqref{transport.reg.1}--\eqref{transport.reg.3}. 
   \item The time for which it exists is $t=1/(2L)= e^{-c\theta_0/\ep}/(2c|||\bv|||_{m,\theta_0,\beta})$ and thus extremely short when $\ep\rightarrow 0$.
However, this time does not depend on the size of the initial datum. It only depends on the norm of $\bv$, on~$\ep$, and on the constant $c$, 
   which in turn depends only on the size of the projection operator $\mathbb{P}$ in the $|\cdot |^\text{C}$-norm. 
   \item Therefore, re-initializing \eqref{transport.reg.1}--\eqref{transport.reg.3}, one can extend the solution as long as $\bv$ exists, i.e., up to the time $T=\theta_0/\beta$. 
  \end{enumerate}
Proposition~\ref{exist.transport.regularized} is thus proven.
\end{proof}

\subsection{The regularized version: uniform in $\ep$ estimates of the regularized solution in the diamond norm}
In the previous subsection, we have constructed the solution $\bureg$ of the regularized transport-pressure problem \eqref{transport.reg.1}--\eqref{transport.reg.3},
with the
time of existence of the solution coinciding with the time of existence of the transporting vector field~$\bv$. 
However, using the estimate of the previous subsection, the norm of $\bureg$, when evaluated in the $|\cdot|^{\text{D}}$ norm, would be~$O(e^{1/\ep})$. 
The goal of this subsection is to give a uniform in $\ep$  estimate of~$\bureg$ by proving the
following statement.

\cole
\begin{proposition}\label{prop:ep_uniform_est}
 Suppose that $\buin\in H^m_{\theta_0,\text{D}} \cap H^m_a$ and 
 $\bv\in H^m_{\theta_0, \beta,\text{D}}\cap H^m_a$
hold
with $|||\bv|||_{m,\theta_0,\beta,a,T}\leq R$ where $T=  (\theta_0 -\bar{\theta})/\beta$ with $0<\bar{\theta}<\theta_0/2$. 
 Then, for $\beta$ large enough, we have the $\ep$-uniform estimate
 \begin{equation}
 |||\bureg|||^{(\gamma)}_{m,\theta_0,\beta,a,T}  \leq K(R)  |||\buin|||_{m,\theta,a}
 \fcomma
 \llabel{EQ128}\end{equation} 
where
 \begin{equation}
   T = \frac{\theta_0 - \bar{\theta}}{\beta}
   \fperiod
   \llabel{EQ129}
 \end{equation} 
\end{proposition}
\colb

To prove the above proposition, we need the following two lemmas, both of which are proven in Section~\ref{section07}. The first lemma gives an $\ep$-uniform analytic estimate of the operator $\Freg$ (integrated in time) 
which involves the weighted  $\gamma$-norm. 
As a consequence one has an estimate on the solution $\bureg$ of the problem \eqref{transport.reg.1}--\eqref{transport.reg.3} or, equivalently,~\eqref{transport.reg.operator.form}. 

\cole
\begin{lemma}[Analytic estimate]\label{analy_estimate}
Suppose that we have
$\bv\in H^m_{\theta_0, \beta,\text{D}}\cap H^m_a$
with $|||\bv|||_{m,\theta,\beta,a,T}<R$ where $T=  (\theta_0 -\bar{\theta})/\beta$ with $0<\bar{\theta}<\theta_0/2$. 
Then
\be
\left|\int_0^t \Freg(\bv, \bureg,\tau)d\tau\right|^{(\gamma)}_{m,\theta_0,\beta}\leq C(R,\beta)
\bigl(|\bureg|^{(\gamma)}_{m,\theta_0,\beta}+ \|\bureg\|_{m,a}\bigr) \fperiod        \label{est:Freg_unif} 
\ee 
Let $\bureg$ be the solution of~\eqref{transport.reg.operator.form}. Then
\be
|\bureg|^{(\gamma)}_{m,\theta_0,\beta}\leq c|\buin|^\text{D}_{m,\theta_0}+
C(R,\beta)\bigl(|\bureg|^{(\gamma)}_{m,\theta_0,\beta}+ \|\bureg\|_{m,a}\bigr)\fperiod      \label{est:bureg_unif} 
\ee
In the above estimates, $C(R,\beta)$ is explicitly given as
\begin{equation}
C(R,\beta)=\cpCau R \frac{2^{\gamma+1}\theta_0^\gamma}{\gamma \beta}\fperiod 
\llabel{EQ136}\end{equation}
\end{lemma}
\colb

Note the inclusion of the Sobolev norm in the 
estimate of the analytic norm of $\bureg$, which
is due to the presence of the strongly nonlocal 
projection operator that involves the value of $\bureg$ away from the boundary.

The next lemma gives an estimate on the Sobolev norm $|\cdot|_{m,a}$ of the solution of \eqref{transport.reg.1}--\eqref{transport.reg.3}. 

\cole
\begin{lemma}[Sobolev estimate]\label{Sobol_estimate}
Suppose that $\buin\in H^m_{\theta_0,\text{D}} \cap H^m_a$ and $\bv\in H^m_{\theta_0, \beta,\text{D}}\cap H^m_a$ with $|||\bv|||_{m,\theta,\beta,a,T}\leq R$,
where $T=  (\theta_0 -\bar{\theta})/\beta$ with $0<\bar{\theta}<\theta_0/2$. 
Then
	\be
   \sup_{0\leq t\leq T} \| \bureg  \|_{m,a}\leq e^{\cSobCau Rt}\|\buin\|_{m,a} +D(R,t) |\bu|^{(\gamma)}_{m,\theta_0,\beta,T} 
\withwith 0<T< \frac{\theta_0- \bar{\theta}}{\beta} 
     \fcomma
     \label{est:Sob_bureg}
	\ee
where
	\begin{equation}
	D(R,t)=  D(R,T) =e^{\cpCau Rt}\sqrt{2\cpCau R}2^{\gamma+1}\frac{\theta_0^{\gamma+1/2}}{\bar{\theta}^{\gamma+1}}\frac{1}{\sqrt{\beta}}\fperiod 
	\llabel{EQ137}\end{equation}
\end{lemma}
\colb

The above estimate holds up to the stopping time~$T$ because  
to estimate the boundary terms arising at $y=a$ in the energy estimate through the use of the Cauchy inequality
one needs enough analyticity (width of the angle larger than~$\bar{\theta}$). 
This means that one gets a Sobolev estimate up to the time $T=(\theta_0-\bar{\theta})/\beta$. 

Using the above two lemmas, we can now prove Proposition~\ref{prop:ep_uniform_est}.

\begin{proof}[Proof of Proposition~\ref{prop:ep_uniform_est}]
Using the estimate \eqref{est:Sob_bureg} in \eqref{est:bureg_unif} and then rearranging the terms, we get
\begin{equation}
  \begin{split}
 &
|\bureg|^{(\gamma)}_{m,\theta_0,\beta}\left(1-C(R,\beta)-C(R,\beta)D(R,\beta)\right)
\\&\indeq
\leq c|\buin|^\text{D}_{m,\theta_0}+e^{cRt} \|\buin \|_{m,a} \leq K'(R) ||| \buin|||_{m,\theta,a}
\fcomma
  \end{split}
\llabel{EQ130}
\end{equation}
with $K'(R)=\max(c,e^{cRT})$.
Observing the expression for $C(R,\beta)$, one sees that, taking $\beta$ sufficiently large, one can have $C$ small.
Note also that $t\leq T\leq \theta_0/\beta$, so that 
$D(R,t)$  can be taken small for $\beta$ large enough. 
Therefore, the term inside the parentheses in the left hand side of the above estimate can be taken larger than~$1/2$. 
This means that   
\begin{equation}
|\bureg|^{(\gamma)}_{m,\theta_0,\beta}<2K'(R)  ||| \buin|||_{m,\theta,a}
\fperiod
\llabel{EQ138}\end{equation}
Inserting the above estimate in \eqref{est:Sob_bureg}, one then immediately gets
\begin{equation}
\|\bureg\|_{m,a}\leq K''(R)  ||| \buin|||_{m,\theta,a}
\fperiod
\llabel{EQ139}\end{equation}
The last two estimates prove Proposition~\ref{prop:ep_uniform_est}
with $K(R)=\max(2K'(R),K''(R))$. 
\end{proof}

The following statement is an immediate consequence of
Proposition~\ref{prop:ep_uniform_est} and the estimate~\eqref{Asano_estimate_I||}. 

\cole
\begin{corollary}
Suppose that the hypotheses of Proposition~\ref{prop:ep_uniform_est} are satisfied. Then $\bureg\in H^m_{\theta_0,\beta',T}\cap H^m_a$ for all $\beta'>\beta$, and  
the $\ep$-uniform estimate 
	\begin{equation}
	|||\bureg|||_{m,\theta_0,\beta',a,T}  \leq (1-\beta/\beta')^{-\gamma} K(R)  |||\buin|||_{m,\theta,a} \qquad \mbox{for} \quad  T = (\theta_0 - \bar{\theta})/\beta
	\llabel{EQ140}\end{equation}
holds.
\end{corollary}
\colb

\colb
\section{Euler equations}
\label{sec06}
Consider the sequence of approximations
  \begin{align}
  \begin{split}
   &
   \pardt \bu^{n+1}
    + \bu^{n}\cdot \bnabla \bu^{n+1}
    + \nabla p^{n+1} = 0
   \\&
   \nabla\cdot \bu^{n+1}=0
   \\&
   \bu^{n+1}(0)=\buin
   \fcomma
  \end{split}
   \label{EQ37}
  \end{align}
with
  \begin{align}
  \begin{split}
   &\bu^{0} = 0
   \\&
   p^{0} = 0
   \fperiod
  \end{split}
   \label{EQ38}
  \end{align}
In particular, $\bu^1=\buin$. 
The above problem may be rewritten as
  \begin{align}
    \begin{split}
		&
		\bu^{n+1}
		= \buin
		- \int_{0}^{t} \mathbb{P} (\bu^{n}\cdot\bnabla  \bu^{n+1})\,ds
	\fperiod
	\end{split}
	\label{EQ41}
\end{align}
Defining
$\bznpo=\bu^{n+1}-\bu^n$  and $\pi^{n+1}=p^{n+1}-p^n$, we obtain the system
  \begin{align} 
	\begin{split}
		&
		\pardt \bznpo
		+ \bu^{n}\cdot \bnabla \bznpo
		+\bzn\cdot\bnabla\bu^n
		+ \nabla \pi^{n+1} = 0
		\\&
		\nabla\cdot \bznpo=0
		\\&
		\bznpo(0)=0
       \fcomma
	\end{split}
	\label{Diff}
\end{align}
which may be rewritten in the corresponding operator form
\be
\bznpo=-\int_0^t \mathbb{P} (\bu^{n}\cdot \bnabla \bznpo)ds -\int_0^t  \mathbb{P} (\bzn\cdot \bnabla \bu^n) ds \label{diff_operatorform}
\fperiod
\ee
Let $\beta>0$, and define the sequence $\left\{\beta_n\right\}$ as 
\begin{equation}
\beta_n=\beta\left(1-\frac{1}{2^{n+1}}\right)
    \comma
    n\in \mathbb{N}\cup \{0\}
\fperiod
\llabel{EQ131}\end{equation} 
Note that  $\beta_n\nearrow \beta$ and  $\beta>\beta_n\geq \beta/2$.  

The main result of this section is the following theorem.

\cole
\begin{theorem}\label{T02}
Suppose that $|||\buin|||_{m,\theta_0,a}<R_0$. Then, assuming that $\beta$ is sufficiently large, $R>4R_0$, and $T<(\theta_0-\bar{\theta})/\beta$ with $\bar{\theta}<\theta_0$, one has:
\begin{enumerate}
	\item (boundedness) $|||\bu^n||||_{m,\theta_0,\beta,a,T}<R$;  
\item (contractiveness) 
				\begin{enumerate} 
					\item $	|||\bznpo|||^{(\gamma)}_{m-1,\theta_0,\beta_n,a,T} \leq L |||\bzn|||^{(\gamma)}_{m-1,\theta_0,\beta_n,a,T}$, 
     				\item  $	|||\bznpo|||^{(\gamma)}_{m-1,\theta_0,\beta,a,T} \leq L |||\bzn|||^{(\gamma)}_{m-1,\theta_0,\beta,a,T}$, 
     			\end{enumerate}
                             with $L<1$.
\end{enumerate} 

\end{theorem}
\colb

In Proposition~\ref{prop:contractiveness}, we provide the contractiveness properties of the sequence $\bu^n$, while the boundedness property is given in Proposition~\ref{P06_improved}. 
We remark that the contractiveness property (2.a) (in the norm where the strip of analyticity shrinks at speed $\beta_n$) is used in the proof of boundedness. 
The contractiveness property (2.b) (in the norm where the strip of analyticity shrinks at speed $\beta$) is applied in Section~\ref{Conclusion} to obtain uniqueness.

In Remark~\ref{rmk.bounds.onbeta}, we summarize the bounds that $\beta$ satisfies to be considered {\em sufficiently large}.

\subsection{Contractiveness in $H^{m-1}_{\theta_0,\beta, a, T}\cap H^{m-1}_a$}
\label{subsec:contractiveness}

The main result of the present subsection is the next proposition, where we show that one can bound the combined $\gamma$-norm 
with derivatives up to order $m-1$ (i.e., the  $|||\cdot|||^{(\gamma)}_{m-1}$-norm)
of $\bznpo$ in terms of the same norm of~$\bzn$. However, this requires that the combined $m$-norm (i.e., the $|||\cdot|||_{m}$-norm) of $\bu^n$  is bounded. 

\cole
\begin{proposition} \label{prop:contractiveness}
Suppose that $|||\bu^n|||_{m,\theta_0,\beta, a,T}<R$.  If $\beta$ is
sufficiently large, the solution $\bznpo$ of the system \eqref{Diff}  satisfies the estimates 
	\begin{align}
	|||\bznpo|||^{(\gamma)}_{m-1,\theta_0,\beta_n,a,T} &\leq A \sqrt{\frac{R}{\beta_n}} |||\bzn|||^{(\gamma)}_{m-1,\theta_0,\beta_n,a,T}    \label{eq:contrativeness}
       \end{align}
and
       \begin{align}
	|||\bznpo|||^{(\gamma)}_{m-1,\theta_0,\beta,a,T}& \leq A \sqrt{\frac{R}{\beta}} |||\bzn|||^{(\gamma)}_{m-1,\theta_0,\beta,a,T}
\fcomma
\label{eq:contrativeness_bis}
	\end{align}
where  $A$ is a constant with an explicit expression given in~\eqref{express_of_A}. 
\end{proposition}
\colb
The condition on $\beta$ can be computed explicitly and is shown in \eqref{first.cond.onbeta} and~\eqref{second.cond.onbeta}.
The proof of the above proposition is provided in Section~\ref{sect:proof_p6.2_l6.3_6.4} and is a direct consequence 
the following two estimates on the analytic and the Sobolev norms.

\cole
\begin{lemma}[The analytic estimate]\label{differences.analytic.est}
Suppose that $|||\bu^n|||_{m,\theta_0,\beta, a,T}<R$. Then, for $0<T=(\theta_0-\bar{\theta})/\beta$,  we have
  \begin{align}
	\begin{split}
		|\bznpo|^{(\gamma)}_{m-1,\theta_0,\beta_n,T} &\leq 
		R \cpCau \frac{2^{\gamma+1}}{\gamma \beta_n} 
		\left(	|\bznpo|^{(\gamma)}_{m-1,\theta_0, \beta_n,T } +   \sup_{0\leq t<T}  \| \bznpo\|_{m-1,a} \right) \\&\indeq+
		R \cpro \frac{2^\gamma\theta_0}{(1-\gamma)\beta_n} \left( |\bzn|^{(\gamma)}_{m-1,\theta_0, \beta_n,T} +   \sup_{0\leq t<T}  \| \bzn\|_{m-1,a}\right) \nonumber 
\fperiod
\end{split}
\end{align}
\end{lemma}

\colb

\cole
\begin{lemma}[The Sobolev estimate]\label{differences.Sob.est}
Suppose that $|||\bu^n|||_{m,\theta_0,\beta, a,T}<R$. Then,   for $0\leq t<T=(\theta_0-\bar{\theta})/\beta$,  we have
  \be 
\|\bznpo\|_{m-1,a}\leq  e^{\cSobCau Rt} \frac{\sqrt{R}}{\sqrt{\beta_n}} \cSobCau \left(  \sup_{0\leq t<T} \| \bzn \|_{m-1,a} +    \frac{|\bznpo|^{(\gamma)}_{m-1,\theta_0,\beta_n,   T }}{\bar{\theta}^{\gamma+1}} \right)
  \fperiod
   \llabel{EQ132}
\ee
\end{lemma}

\colb
The proofs of the two lemmas are given in Section~\ref{sect:proof_p6.2_l6.3_6.4}.

\subsection{Boundedness of the sequence $\{\bu^n \}$ in the $||| \cdot|||_{m-1}$-norm}

The next proposition establishes the boundedness of the sequence~$\{\bu^n \}$. Note, however, that the boundedness is in the $(m-1)$-norm. In the next subsection,
this result
is
bootstrapped to the $m$-norm.

\cole
\begin{proposition}
\label{P06}
Suppose that $|||\buin|||_{m,\theta_0,a}<R_0$ and $|||\bu^n||||_{m,\theta_0,\beta,a,T}<R$ with $R>4R_0$. 
Then, for $\beta$ sufficiently large, we have $|||\bu^{n+1}||||_{m-1,\theta_0,\beta,a,T}<R$.
\end{proposition}   
\colb

\begin{proof}[Proof of Proposition~\ref{P06}]
First, we observe that using the property  \eqref{Asano_estimate} and the expression of $\beta_n$ in terms of $\beta$,  we may write
	\begin{equation}
	|\bznpo|^\text{D}_{m-1,\theta_0,\beta,T}\leq (1-\beta_n/\beta)^{-\gamma}  	|\bznpo|^{(\gamma)}_{m-1,\theta_0,\beta_n,T}=
	2^{\gamma(n+1)} 	|\bznpo|^{(\gamma)}_{m-1,\theta_0,\beta_n,T}\fperiod 
	\llabel{EQ142}\end{equation}
	We can thus use the contractiveness property \eqref{eq:contrativeness} and estimate
  \begin{align}
	\begin{split} 
|\bznpo|^\text{D} _{m-1,\theta_0,\beta,T}  &\leq 2^{\gamma(n+1)} 	|\bznpo|^{(\gamma)}_{m-1,\theta_0,\beta_n,T}    \leq  2^{\gamma(n+1)}  |||\bznpo |||^{(\gamma)}_{m-1,\theta_0,\beta_n,T} \nonumber \\
& \leq  2^{\gamma(n+1)}  A\sqrt{\frac{R}{\beta_n}}|||\bzn |||^{(\gamma)}_{m-1,\theta_0,\beta_n,T} 
\\&
 \leq  2^{\gamma(n+1)} \left(  A\sqrt{\frac{R}{\beta_n}} \right)^n  ||| \bz^1 |||^{(\gamma)}_{m-1,\theta_0,\beta_n,T} \nonumber \\
& \leq  2^{\gamma} \left(  2^\gamma A \sqrt{\frac{R}{\beta/2}} \right)^n  ||| \bz^1 |||^{(\gamma)}_{m-1,\theta_0,\beta_1,T}
\fcomma
\end{split}
\end{align}
where the last inequality is justified by using $\beta_n\geq \beta/2 $ and $\beta_1>\beta_n$. 
Now, observe that $\bz_1=\bu^1-\bu^0=\buin$, and impose
\begin{equation} 
\beta> 2^{2\gamma+1}A^2 R
\fcomma
\label{third.cond.onbeta}\end{equation}
so that $\lambda=\sqrt{2}2^\gamma A\sqrt{R}/\sqrt{\beta}<1$, to get 
\begin{equation}
|\bu^{n+1}-\bu^n|^\text{D} _{m-1,\theta_0,\beta,T}\leq 2^\gamma \lambda^n  |||\buin|||_{m-1,\theta_0,a}=2^\gamma\lambda^n R_0\fperiod 
\label{EQ143}\end{equation}
Using $\bu^{n+1}=\bu^{n+1}-\bu^{n} + \bu^{n}-\bu^{n-1}+\cdots+\bu^{1}-\bu^{0}$, we obtain
\begin{equation}
|\bu^{n+1}|^\text{D} _{m-1,\theta_0,\beta,T}\leq 2^\gamma R_0(\lambda^n+\cdots+1)\leq \frac{2^\gamma R_0}{1-\lambda}\fperiod 
\llabel{EQ144}\end{equation}
Imposing $\lambda< 1-4R_0/R$, i.e.,
\begin{equation}
\beta > 2^{2\gamma+1} \left(\frac{R}{R-4R_0}\right)^2 A^2 R
\fcomma
\label{fourth.cond.onbeta}\end{equation}
we derive a bound on the analytic norm of  ${\bu^{n+1}}$, which reads
\begin{equation}
|\bu^{n+1}|^\text{D}_{m-1,\theta_0,\beta,T}<\frac{R}{2}
\fperiod
\llabel{EQ146}\end{equation}
Through a standard energy estimate and applying the above bound on\\$|\bu^{n+1}|^\text{D} _{m-1,\theta_0,\beta,T}$, it is easy to show that 
\begin{equation}
\|\bu^{n+1}\|_{m-1,a,T}<\frac{R}{2}\fcomma 
\llabel{EQ147}\end{equation}
so that the bound $|||\bu^{n+1}|||_{m-1,\theta_0,\beta,a,T}<R $ follows. 
\end{proof}

\subsection{Boundedness of the sequence $\bu^n$ in the $|||\cdot|||_m$-norm}

We have proven the boundedness of $\bu^{n+1}$ in the $(m-1)$-norms. 
We now prove an estimate on  $\bu^{n+1}$ in the $m$-norms. 
\cole
\begin{proposition}
	\label{P06_improved}
Suppose that $|||\buin|||_{m,\theta_0,a}<R_0$ and $|||\bu^n||||_{m,\theta_0,\beta,a,T}<R$ with $R>4R_0$. 
Then, for $\beta$ sufficiently large,  $|||\bu^{n+1}||||_{m,\theta_0,\beta,a,T}<R$.
\end{proposition}   
\colb

To show a bound of the $m$-norm, we first prove an estimate of the Sobolev norm. 
This is accomplished with a standard energy estimate in~$H^m_a$. However, this leads to the appearance of $m$ derivatives of $\bu^{n+1}$
at $y=a$.
The estimate of these terms is achieved by using the analyticity of~$\bu^{n+1}$. To be more precise, we use fact that $\bu^{n+1}  \in H^{m-1}_{\bar{\theta}}$
and  the Cauchy estimate twice to get a bound of the $D^m$-derivatives at $y=a$ in terms of the norm in~$H^{m-1}_{\bar{\theta}}$. 
The main technical step is contained in the following lemma. 

\cole
\begin{lemma} \label{lemma:two_times_Caucht}
Let $f\in H^{l}_{\theta}\cap H^l_a$.  Then, if $i\leq l$ and $\,\bar{\theta}<\theta/2$, we have
the estimate
	\begin{equation}
	\int \bigl| \pardx^{l-i}\pardy^{i+1} f(x, y=a)\bigr|^2dx\leq \cCautwo \left( \frac{|f|^\text{D}_{l,2 \bar{\theta}}}{\bar{\theta}^2}\right)^2
	\fperiod
	\llabel{EQ148}\end{equation} 
\end{lemma}
\colb

\begin{proof}[Proof of Lemma~\ref{lemma:two_times_Caucht}]
For simplicity, we provide the estimate when $l=0$.
Then
\begin{align}
	\begin{split}
		\left| \pardy f(x, y=a)\right| &\leq \sup_{|y-a|\leq \bar{\theta}/2}\left|   \pardy f(x, y)\right|  \nonumber 
\\&
   \leq \left(\int_{|y-a|\leq \bar{\theta}/2} |  \pardy f|^2 dy \right)^{1/2}
   +  \left(\int_{|y-a|\leq \bar{\theta}/2} |   \pardyy f|^2 dy \right)^{1/2} \nonumber \\
& \leq \frac{| f|^\text{D}_{\bar{\theta}} }{\bar{\theta}} + \frac{|  \pardy f|^\text{D}_{\bar{\theta}} }{\bar{\theta}} \nonumber 
\leq \frac{| f|^\text{D}_{\bar{\theta}} }{\bar{\theta}} + \frac{|   f|^\text{D}_{\bar{2\theta}} }{\bar{\theta}^2} \nonumber 
\leq c  \frac{| f|^\text{D}_{ \bar{2\theta}} }{\bar{\theta}^2} \fperiod 
\end{split}
\end{align}
	When $l>0$, the estimate is analogous. 
\end{proof}

\begin{proof}[Proof of Proposition~\ref{P06_improved}]
We now evaluate the $m$-Sobolev norm of~$\bu^{n+1}$.
Consider the first equation in~\eqref{EQ37}. Taking $l$ derivatives, where $0\leq l\leq m$, and then applying the scalar product in $L^2$ with $D^l \bu^{n+1}$,
we obtain
\begin{equation}
  \begin{split}
\frac{1}{2}\frac{d}{dt} \| \bu^{n+1}\|^2_{m,a}
&\leq   \cSob \|\bu^n\|_{m,a}\|\bu^{n+1}\|^2_{m,a} +
\int  v^n \left| D^m\bu^{n+1}\right|^2_{y=a} dx
\\&\indeq
+
\sum_l \int \left( D^l p^{n+1} \bn \cdot D^l \bu^{n+1}\right)_{y=a} dx
\fperiod 
 \end{split}
\llabel{EQ149}
\end{equation}
Then, using Lemma~\ref{lemma:two_times_Caucht}, we have
$$
\int  v^n \left| D^m\bu^{n+1}\right|^2_{y=a} dx\leq \cCautwo |\bu^n|^\text{D}_{m,\bar{\theta}} \left(  \frac{|\bu^{n+1}|^\text{D}_{m-1,\bar{\theta}}}{\bar{\theta}^2}\right)^2
\fperiod
$$
To bound the term resulting from the pressure, we proceed as in the estimates~\eqref{est::BT2} and~\eqref{est::pressure} and write
\begin{align}
  &	\sum_l \int \left( D^l p^{n+1} \bn \cdot D^l \bu^{n+1}\right)_{y=a} dx
					\leq \sum_l\left(\int _{\mathbb{R}} \left(D^l p^{n+1} \right)^2_{y=a}  dx\right)^{1/2} \left(\int _{\mathbb{R}} \left(D^l \bu^{n+1} \right)^2_{y=a}  dx\right)^{1/2} \nonumber \\
	& \leq \cSob \|p^{n+1}\|_{H^{m+1}(a-\bar{\theta},a+\bar{\theta})} \cdot \cCautwo \frac{|\bu^{n+1}|^\text{D}_{m-1,\bar{\theta}} }{\bar{\theta}^2}  \nonumber \\
	&  \leq \cSobCau \left( \| \bu^n \|_{H^m(a-\bar{\theta},a+\bar{\theta})}  \| \bu^{n+1} \|_{H^m(a-\bar{\theta},a+\bar{\theta})} +\|p\| _{L^2}\right) 
	 \frac{|\bu^{n+1}|^\text{D}_{m-1,\bar{\theta}} }{\bar{\theta}^2}  \nonumber \\
	 &  \leq \cSobCau\left(R \frac{|\bu^{n+1}|^\text{D}_{m-1,\bar{\theta}} }{\bar{\theta}} +  \|\bu^n\|_{H^2} \|\bu^{n+1}\|_{H^2} \right)  
	 \frac{|\bu^{n+1}|^\text{D}_{m-1,\bar{\theta}} }{\bar{\theta}^2}  \nonumber \\
	& \leq \cSobCau \left( R\frac{|\bu^{n+1}|^\text{D}_{m-1,\bar{\theta}} }{\bar{\theta}} + R  \left(\|\bu^{n+1}\|_{m-1,a} +|\bu^{n+1}|^\text{D}_{m-1,\bar{\theta}}\right)\right)  
	 \frac{|\bu^{n+1}|^\text{D}_{m-1,\bar{\theta}} }{\bar{\theta}^2}  \nonumber \\
	 & \leq \cSobCau R \left( \left(  \frac{|\bu^{n+1}|^\text{D}_{m-1,\bar{\theta}} }{\bar{\theta}^2}    \right)^2 +  \|\bu^{n+1}\|^2_{m-1,a} \right) \nonumber 
\fperiod
\end{align}
Then,
\begin{equation}
\frac{1}{2}\frac{d}{dt} \| \bu^{n+1}\|^2_{m,a}    \leq     \cSobCau R \|\bu^{n+1}\|^2_{m,a} +
\cSobCau R\left(  \frac{|\bu^{n+1}|^\text{D}_{m-1,\bar{\theta}}}{\bar{\theta}^2}\right)^2
\comma 0\leq t\leq T< (\theta_0-\bar{\theta})/\beta
\fcomma
\llabel{EQ150}\end{equation}
which, by using
$|||\bu^{n+1}|||_{m-1,\theta,\beta,a,T}<R$
and the Gronwall Lemma, we get
$$
\|\bu^{n+1}\|_{m,a,T}<\frac{R}{2}
\fcomma
$$
provided
\be 
\beta>R\frac{4\cSobCau  \theta_0 e^{2\cSobCau RT}}{\bar{\theta}^4} \label{fifth.cond.onbeta}
\fperiod
\ee

With the above bound on $\|\bu^{n+1}\|_{m,a,T}$ and using the equation~\eqref{EQ41}, one can immediately derive the analogous bound on the 
analytic norm, obtaining
$$
|\bu^{n+1}|^\text{D}_{m,\theta_0,\beta,T}<\frac{R}{2}
\fcomma
$$
which completes the proof of Proposition~\ref{P06_improved}.
\end{proof}

Propositions~\ref{P06_improved} and~\ref{prop:contractiveness}, stating the boundedness and the contractiveness of the sequence, respectively, 
prove Theorem~\ref{T02}.

\begin{remark}[Bounds on $\beta$] \label{rmk.bounds.onbeta}
Here, we summarize the bound that $\beta$ satisfies to be large enough so that Proposition~\ref{prop:contractiveness} 
	(contractiveness of $\bu^n$ in the $	|||\cdot |||^{(\gamma)}_{m-1,\theta_0,\beta_n,a,T}$-norm)
	and Proposition~\ref{P06_improved} (boundedness of $\bu^n$ in the $|||\cdot |||_{m,\theta_0,\beta,a,T}$-norm) hold. 
Namely, $\beta$ has to satisfy the bounds 
\eqref{first.cond.onbeta}, \eqref{second.cond.onbeta}, \eqref{third.cond.onbeta}, \eqref{fourth.cond.onbeta}, and~\eqref{fifth.cond.onbeta}.
\end{remark}

\cole
\begin{corollary}[Cauchy property in the $|||\cdot |||_{m-1,\theta_0,\beta,a,T }$-norm] \label{C01}
Suppose that we have $|||\buin|||_{m,\theta_0,a}<R_0$. Then, assuming that $\beta$ is large enough, $R>4R_0$, and $T<(\theta_0-\bar{\theta})/\beta$ with $\bar{\theta}<\theta_0$, we have
 \be	|||\bznpo|||_{m-1,\theta_0,\beta,a,T} \leq c \lambda^n
 \fcomma
 \label{EQ143_bis}
 \ee
with  $\lambda<1$.
\end{corollary}
\colb

\begin{proof}[Proof of Corollary~\ref{C01}]
	It is sufficient to look at the estimate \eqref{EQ143} to see that one has the inequality \eqref{EQ143_bis} for  the $|\cdot|^\text{D}_{m-1,\theta_0,\beta,T} $-norm. 
	We already have the contractiveness in the $\|\cdot\|_{m-1,a}$-norm, and the estimate \eqref{EQ143_bis} is therefore proven. 
\end{proof}

\subsection{Conclusion of the proof of Theorem~\ref{T01} on the existence and uniqueness for the Euler solution}\label{Conclusion}

Given $\bu^n$, Theorem~\ref{T.linear} shows that one can define $\bu^{n+1}$ as the solution of the problem~\eqref{EQ41}. 
The part~(1) of Theorem~\ref{T02} ensures that $\bu^n$ remains bounded in the norm $|||\cdot|||_{m,\theta_0,\beta, a, T}$ within a ball of radius 
$R>4R_0$, where $R_0$ is the size of the initial datum.
This ensures that the sequence $\{\bu^n\}$ is well-defined. 

Corollary~\ref{C01} ensures that $\{\bu^n\}$ is Cauchy in the norm
$|||\cdot |||_{m-1, \theta_0, \beta,a,T}$ and, therefore, it converges
to $\bu$ in the same norm. Clearly, $\bu$ solves Euler equations.

To obtain uniqueness, assume that $\bu$ and $\bv$ are two solutions
of the Euler equations. Note that both have enough regularity to apply
the standard Euler uniqueness. We namely subtract the equations
for $\bu$ and $\bv$ and test the resulting equation
with $\bu-\bv$. We omit further details.

\section{Proofs of the analytic estimates} 
\label{section07}

\subsection{Projection operator estimate: proof of Proposition~\ref{prop_projection} } \label{subsect_projection_est}

In the explicit expressions \eqref{Pplus_x} and~\eqref{Pplus_y}, one recognizes that we need to show the analyticity in $D_\theta$ and the estimates for
 terms of the form
\begin{equation}
  \begin{split}
  &
 \int_y^\infty \,dy' |\xi|e^{|\xi|(y-y')} f(\xi,y'), \quad  \int_0^y \,dy' |\xi| e^{-|\xi|(y-y')} f(\xi,y'),
  \\&\indeq
   \quad  \int_0^y \,dy' |\xi| e^{-|\xi|(y+y')} f(\xi,y')
   \fcomma
  \end{split}
   \llabel{EQ63}
   \end{equation}
for $y\in D_\theta$.
Therefore, the proof of Proposition~\ref{prop_projection} is an immediate consequence of the following lemma.

\cole
\begin{lemma}
\label{est_proj}
Assume that $f\in H^m_{\theta}\cap H^m_a$.
Then $ \int_y^\infty \,dy' |\xi|e^{|\xi|(y-y')} f(\xi,y')$,
\\
$ \int_0^y \,dy' |\xi|e^{-|\xi|(y-y')} f(\xi,y')$, and 
$ \int_0^y \,dy' |\xi| e^{-|\xi|(y+y')} f(\xi,y')$ belong to~$H^m_{\theta}$.   
Moreover, we have the estimates
  \begin{align}
   &
   \left| \int_y^\infty \,dy' |\xi|e^{|\xi|(y-y')} f(\xi,y')\right|_{m,\theta} \leq c |f|_{m,\theta}
   \fcomma\label{est_proj_1}
   \\&
   \left| \int_0^y \,dy' |\xi| e^{-|\xi|(y-y')} f(\xi,y')\right|_{m,\theta} \leq c |f|_{m,\theta}\fcomma \label{est_proj_2}
   \\&
   \left| \int_0^y \,dy' |\xi| e^{-|\xi|(y+y')} f(\xi,y')\right|_{m,\theta} \leq c |f|_{m,\theta}
  \fperiod
  \label{est_proj_3}
  \end{align}
\end{lemma}
\colb

To prove the above lemma, we need to show first that the above three terms are $y$-analytic in~$D_\theta$.
Given that  $D_\theta$ is open when $y\in D_\theta$,
a whole neighborhood of $y$ is contained  in~$D_\theta$. 
This allows us to compute the complex derivative of $\mathbb{P} \bu$.
For example, one easily gets
\begin{equation}
\pardy \int_y^\infty \,dy' |\xi|e^{|\xi|(y-y')} u(\xi,y') = -|\xi|u(\xi,y) -|\xi|^2 \int_y^\infty \,dy' e^{|\xi|(y-y')} u(\xi,y')
   \fcomma
   \llabel{EQ64}
   \end{equation}
which, given the exponential decay  in the $\xi$-variable of $u(\xi,y)$,  shows that the terms to be estimated are holomorphic in~$D_\theta$. 

Second, we need to estimate the norm of the three terms in~$H^m_{\theta}$. We shall see how to obtain the inequality \eqref{est_proj_2} and comment on the others.
Given $\theta\in(-\pi/2,\pi/2)$, denote
$\Gamma_1(\theta)=(0,1+i\tan \theta)$, i.e, $\Gamma_1(\theta)$ is the line in the complex plane
between $0$ and $1+i\tan \theta$; also,
let $\Gamma_2=(1+i\tan \theta, 1+\theta)$.
We  fix $\theta$, and we estimate the integral
  \begin{equation}
   I
   =
   \int_{-\infty}^{\infty}\,d\xi
    \int \,|\,dy|
   \left|
    \int_{0}^{y}
    \,dy'
    |\xi| e^{-|\xi|(y'-y_0)} f(\xi,y')
   \right|^2
   e^{2\rho_{\theta}(y)|\xi|}
   \fcomma
   \label{EQ106}
  \end{equation}
for $y\in \Gamma_1(\theta)$ first.
To treat it,
we parameterize $y$ and $y'$ by
$y=\alpha+i\alpha \tan\theta$
and
$y'=\alpha'+i\alpha' \tan\theta$,
respectively, and estimate
  \begin{align}
  \begin{split}
   I
    &=
    \int_{-\infty}^{\infty}
     \,d\xi
     \int_{0}^{1}\,d\alpha\sqrt{1+\tan^{2}\theta}
     \biggl|
      \int_{0}^{\alpha}
       \,d\alpha'
       |\xi| 
       (1+i\tan\theta)
       e^{-|\xi|(\alpha-\alpha')(1+i\tan\theta)}
    \\&\indeq\indeq\indeq\times
       f(\xi,\alpha'+i\alpha'\tan\theta)
  \biggr|^2 e^{\theta|\xi|}
   \\&
   \leq
  c ({1+\tan^{2}\theta})
   \int_{-\infty}^{\infty}
     \,d\xi
     \int_{0}^{1}\,d\alpha
     \left|
      \int_{0}^{\alpha}
       \,d\alpha'
       |\xi| 
       e^{-|\xi|(\alpha-\alpha')}
       |f(\xi,\alpha'+i\alpha'\tan\theta)|
  \right|^2 e^{\theta|\xi|}
  \\&
  =
   c({1+\tan^{2}\theta})
   \int_{-\infty}^{\infty}
     \,d\xi
     \int_{0}^{1}\,d\alpha
     \left|
      \int_{0}^{\alpha}
       \,d\alpha'
       |\xi| 
       e^{-|\xi|(\alpha-\alpha')}
       |\tilde f(\xi,\alpha')|
  \right|^2 e^{\theta|\xi|}
  \fcomma
  \end{split}
   \label{EQ107}
  \end{align}
where we abbreviate
$\tilde f(\xi,\alpha')=f(\xi,y')$ (where always $\alpha'=\Re {y'}$). For simplicity of
notation, we always assume
$\tilde f(\xi,\alpha')=0$ if $\alpha'<0$ or $\alpha'\geq 1+\theta$.
Also, let $K(\beta,\xi)=|\xi| e^{-\beta|\xi|}\chi_{(0,\infty)}(\beta)$.
Then,
  \begin{align}
  \begin{split}
   I
   &\lec_{\theta}
   \int_{-\infty}^{\infty}
    \,d\xi
     \int_{0}^{1}\,d\alpha
     \left|
      \int_{0}^{\alpha}
       \,d\alpha'
        K(\alpha-\alpha',\xi)
       |\tilde f(\xi,\alpha')|
  \right|^2 e^{\theta|\xi|}
  \\&
  =
   \int_{-\infty}^{\infty}
    \,d\xi
     \int_{0}^{1}\,d\alpha
     \left|
      \int_{-\infty}^{\infty}
       \,d\alpha'
        K(\alpha-\alpha',\xi)
       |\tilde f(\xi,\alpha')|
  \right|^2 e^{\theta|\xi|}
  \\&
  \lec
   \int_{-\infty}^{\infty}
    \,d\xi
  \Vert K(\cdot,\xi)\Vert_{L^{1}}^2
  \int_{0}^{1}\,d\alpha'|\tilde f(\xi,\alpha')|^2
   e^{\theta|\xi|}
   \\&
   \lec
   \int_{-\infty}^{\infty}
    \,d\xi
   \int_{0}^{1}\,d\alpha'|\tilde f(\xi,\alpha')|^2
   e^{\theta|\xi|}
  \lec_{\theta}
   \int_{-\infty}^{\infty}
    \,d\xi
   \int_{0}^{1}\,|\,dy'| |f(\xi,y')|^2
   e^{\theta|\xi|}
   \lec |f|_{\theta}
   \fperiod
  \end{split}
   \label{EQ109}
  \end{align}
Next, we estimate the integral \eqref{EQ106} when $y\in \Gamma_2(\theta)$.
The path integral in $y'$ is divided into
two integrals corresponding to the lines $(0,1+i\tan\theta)$ and $(1+i\tan\theta,y)$; we denote
the corresponding quantities by $I_1$ and~$I_2$.
They are bounded in a similar manner; we show the details only for~$I_2$.
By parameterizing
$y=\alpha+i(1+\theta-\alpha)\tan \theta/\theta$ and
$y'=\alpha'+i(1+\theta-\alpha')\tan \theta/\theta$, we have
  \begin{align}
  \begin{split}
   I_2
    &\lec
    \int_{-\infty}^{\infty}
     \,d\xi
     \int_{\Gamma_2(\theta)}\,|dy|
     \left|
      \int_{1+i\tan\theta}^{y}
       \,dy' |\xi|
       e^{-|\xi|(y-y')}
       f(\xi,y')
  \right|^2 e^{2\rho_\theta(y)|\xi|}
   \\&
   \lec_{\theta}
    \int_{-\infty}^{\infty}
     \,d\xi
     \int_{1}^{1+\theta}\,d\alpha
     \left|
      \int_{1}^{\alpha}
       \,d\alpha'
       |\xi|e^{-|\xi|(\alpha-\alpha')}
       |\tilde f(\xi,\alpha')|
  \right|^2
  e^{|\xi|(1+\theta-\alpha)}
  \\&
  \lec
    \int_{-\infty}^{\infty}
     \,d\xi
     \int_{1}^{1+\theta}\,d\alpha
     \left|
      \int_{1}^{\alpha}
       \,d\alpha'
       |\xi|e^{-|\xi|(\alpha-\alpha')}
       |\tilde f(\xi,\alpha')|
        e^{|\xi|(1+\theta-\alpha')/2}
  \right|^2
   \lec
   |f|_{\theta}^2
  \fcomma
  \end{split}
  \label{EQ108}
  \end{align}
where we used $\alpha'\leq \alpha$ in the third inequality; 
in the step, we used Young's inequality, concluding similarly
to~\eqref{EQ109}.

Next, we estimate a more difficult term involving the integral in $y'$
from $y$ to~$\infty$, i.e., we intend to bound
  \begin{equation}
   I
   =
   \int_{-\infty}^{\infty}\,d\xi
   \int_{\Gamma_1\cup\Gamma_2} \,|\,dy|
   \left|
    \int_{y}^{\infty}
    \,dy'
    |\xi| e^{|\xi|(y-y')} f(\xi,y')
   \right|^2
   e^{2\rho_{\theta}(y)|\xi|}
  \fperiod
   \label{EQ110}
  \end{equation}
The term $I$ is less than or equal to 
$I_1+I_2$, where $I_1$ and $I_2$ correspond to the integrals in $y$ over $\Gamma_1(\theta)$
and $\Gamma_2(\theta)$, respectively.
Next, we have $I_1\leq I_{11} + I_{12} + I_{13} $, where the three integrals correspond
to integration in $y'$ over $(y,1+i\tan\theta)$,
$(1+i\tan\theta, 1+\theta)$, and $(1+\theta,\infty)$.
Similarly, $I_2\leq I_{21}+I_{22}$, where the two terms correspond to
$y'$ integrations in $(y,1+\theta)$ and $(1+\theta,\infty)$, respectively.
The integrals are estimated in a similar manner; here we show how to handle
$I_{21}$ and~$I_{12}$. 

We first estimate $I_{21}$ when  $y\in \Gamma_2(\theta)$ and $y'\in\Gamma(\theta,y)= (y,1+\theta)$. 
We parameterize the two paths as follows: 
$\Gamma_2(\theta)=\left\{  y=\alpha +i(1+\theta-\alpha)\tan\theta/\theta, \;\alpha\in(1,1+\theta)    \right\}$ 
and $\Gamma(\theta,y)=\left\{  y'=\alpha' +i(1+\theta-\alpha')\tan\theta/\theta, \;\alpha'\in(\alpha,1+\theta)    \right\}$. 
Then,
  \begin{align}
	\begin{split}
		I_{21}
		&=
   \int_{-\infty}^{\infty}
\,d\xi
\int_{\Gamma_2(\theta)}
|\,d y|
\left|
\int_{y}^{1+\theta}
\,dy'
|\xi|
e^{|\xi|(y-y')}
f(\xi,y')
e^{|\xi|\rho_{\theta}(y)}
\right|^2
\\&
=   \int_{-\infty}^{\infty}
\,d\xi
\int_{1}^{1+\theta}
\,d\alpha
\sqrt{1+(\tan\theta/\theta)^2}
\\&\indeq\indeq\indeq\indeq\times
\left|    
\int_{\alpha}^{1+\theta} 
\,d\alpha'(1-i\tan\theta/\theta)
|\xi| 
e^{|\xi|(\alpha-\alpha')(1+i\tan\theta/\theta)}
    |\tilde f(\xi,\alpha')|
    e^{|\xi|\rho_{\theta}(\alpha)}
\right|^2
   \\&
\lec_{\theta}
\int_{-\infty}^{\infty}
\,d\xi
\int_{1}^{1+\theta}
\,d\alpha
\left|
\int_{\alpha}^{1+\theta}
\,d\alpha'
|\xi|
e^{-|\xi|(\alpha'-\alpha)}
|\tilde f(\xi,\alpha')|
e^{|\xi|\rho_{\theta}(\alpha)}
\right|^2
\fperiod
  \end{split}
\label{EQ111}
\end{align}
Given that $\Re{y} \leq 1, \Re{y'}\geq 1$, we have $\rho_{\theta}(\alpha)=(1+\theta-\alpha)/2$ and $\rho_{\theta}(\alpha')=(1+\theta-\alpha')/2$, so that  
$		-(\alpha'-\alpha)
		-\rho_{\theta}(\alpha')
		+\rho_{\theta}(\alpha)
		=
		-(\alpha'-\alpha)
		-(1+\theta-\alpha')/{2}
		+(1+\theta-\alpha)/{2}
		= (\alpha-\alpha')/{2}. 
$
Therefore, 
 \begin{align}\label{EQ146}
	\begin{split}
		I_{21}
		&\lec_{\theta}
		\int_{-\infty}^{\infty}
		\,d\xi
		\int_{1}^{1+\theta}
		\,d\alpha
		\left|
		\int_{\alpha}^{1+\theta}
		\,d\alpha'
		|\xi|
		e^{-|\xi|(\alpha'-\alpha)/2}
		|\tilde f(\xi,\alpha')|
		e^{|\xi|\rho_{\theta}(\alpha')}
		\right|^2
		  \\&
		\lec
		\int_{-\infty}^{\infty}
		\,d\xi
		\int_{0}^{\infty}
		\,d\alpha
		|\tilde f(\xi,\alpha)|^2
		e^{2|\xi|\rho_{\theta}(\alpha)}
		\\&
		\lec_{\theta}
		\int_{-\infty}^{\infty}
		\,d\xi
		\int_{0}^{\infty}
		\,|dy|
		|f(\xi,y)|^2
		e^{2|\xi|\rho_{\theta}(y)}
		\lec
		|f|_{\theta}^2
	\fperiod
\end{split}
\end{align}
We now pass to $I_{12}$ when  $y\in \Gamma_1(\theta)$ and $y'\in\Gamma(\theta,y)= (1+i\tan\theta,1+\theta)$.
We parameterize the two paths as $\Gamma_1(\theta)=\left\{  y=\alpha +i\alpha\tan\theta, \;\alpha\in(0,1)    \right\}$ 
and $\Gamma(\theta,y)]=\left\{  y'=\alpha' +i(1+\theta-\alpha')\tan\theta/\theta, \;\alpha'\in(1,1+\theta)    \right\}$. Then,
  \begin{align}
	\begin{split}
		I_{12}
		&=
		\int_{-\infty}^{\infty}
		\,d\xi
		\int_{\Gamma_1(\theta)}
		|\,d y|
		\left|
		\int_{1+i\tan\theta}^{1+\theta}
		\,dy'
		|\xi|
		e^{|\xi|(y-y')}
		f(\xi,y')
		e^{|\xi|\rho_{\theta}(y)}
		\right|^2
		\\&
		\lec_{\theta}
		\int_{-\infty}^{\infty}
		\,d\xi
		\int_{0}^{1}
		\,d\alpha
		\left|
		\int_{1}^{1+\theta}
		\,d\alpha'
		|\xi|
		e^{-|\xi|(\alpha'-\alpha)}
		|\tilde f(\xi,\alpha')|
		e^{|\xi|\rho_{\theta}(\alpha)}
		\right|^2
		\fperiod
		\end{split}
		\label{EQ120}
\end{align}
Given that  $\Re{y} \leq 1$ and  $\Re{y'}\geq 1$, we have $\rho_{\theta}(\alpha)=\theta/2$ and $\rho_{\theta}(\alpha')=(1+\theta-\alpha')/2$, and we can write  
$
		-(\alpha'-\alpha)
		-\rho_{\theta}(\alpha')
		+\rho_{\theta}(\alpha)
		=
		-(\alpha'-\alpha)
		-(1+\theta-\alpha')/{2}
		+{\theta}/{2}
		= (\alpha-\alpha')/{2} + (\alpha-1)/{2}
		\leq (\alpha-\alpha')/{2}.
$
The rest of the estimate proceeds as in \eqref{EQ146} with minor modifications. Namely,
 \begin{align}
	\begin{split}
		I_{12}
		&\lec_{\theta}
		\int_{-\infty}^{\infty}
		\,d\xi
		\int_{0}^{1}
		\,d\alpha
		\left|
		\int_{1}^{1+\theta}
		\,d\alpha'
		|\xi|
		e^{-|\xi|(\alpha'-\alpha)/2}
		|\tilde f(\xi,\alpha')|
		e^{|\xi|\rho_{\theta}(\alpha')}
		\right|^2
		\\&
		\lec
		\int_{-\infty}^{\infty}
		\,d\xi
		\int_{0}^{\infty}
		\,d\alpha
		|\tilde f(\xi,\alpha)|^2
		e^{2|\xi|\rho_{\theta}(\alpha)}
		\\&
		\lec_{\theta}
		\int_{-\infty}^{\infty}
		\,d\xi
		\int_{0}^{\infty}
		\,|dy|
		|f(\xi,y)|^2
		e^{2|\xi|\rho_{\theta}(y)}
		\lec
		|f|_{\theta}^2
		\fperiod
	\end{split}
\end{align}
This concludes the proof of the Lemma~\ref{est_proj} and, therefore, of Proposition~\ref{prop_projection}. 
\colb
\subsection{Proof of Proposition~\ref{P01}}

To prove the proposition, we first analyze the case when  
 $u$ and $v$ are scalar functions; in this case we denote them by $f$ and $g$, respectively. 
 Therefore, we need to estimate terms of the type \eqref{est_proj_1}--\eqref{est_proj_3} with $f \pardx g$ instead of~$f$. How to treat derivatives with 
 respect to $y$ is addressed further below. 
We shall give a rather detailed proof for the first term, when integration in $y'$ goes from $y$ to~$\infty$.  Then,
we claim
  \be
     \left| \int_y^\infty \,dy' |\xi| e^{|\xi|(y-y')} \Bigl(f(x,y')\pardx g(x,y')\Bigr)^{\wedge}\right|_{m,\theta''} \leq c |f|_{m,\theta''}\frac{|g|_{m,\theta}}{\theta-\theta''}\, , \quad \theta''<\theta
     \fperiod
     \label{EQA96}
  \ee
Using the same ideas, one can estimate the other terms, where the
integration in $y'$ is over $(0,y)$.
To prove the estimate \eqref{EQA96}, we compute the $L^2$ norm along the path $\Gamma(\theta')$ (see \eqref{EQ86}) with $\theta'<\theta''$. 
We shall consider in detail the case the part of this path when $y\in(0,1+i\tan\theta')$ and begin estimating the term, denoted by~$I$, when $y'\in (y, 1+i\tan\theta')$. 
The case when $y'\in(1+i\tan\theta, 1+\theta)$ is analyzed below, see \eqref{EQ93}, while the case when $y'\in(1+\theta, \infty)$ is simpler and is omitted. Thus the term~$I$ reads
   \begin{align}
  \begin{split}
   &
   I=
   \int_{-\infty}^{\infty}
   \,d\xi
   \int_{0}^{1+i\tan\theta'}
   \,dy
   \\&\indeq\indeq\indeq\indeq\times
   \left|
    \int_{y}^{1+i\tan\theta'}
    \,dy'
    |\xi|
    e^{|\xi|e^{|\xi|(y-y')}}
    |\xi|^{m}
    \int^{\infty}_{-\infty}\,d\eta
    f(\xi-\eta)
    g(\eta)
    |\eta|
    e^{\rho_{\theta'}(y) |\xi|}
   \right|^2
  \\&\indeq
   \lec
   \int_{-\infty}^{\infty}\,d\xi
   \int_{0}^{1}
   \,d\alpha
   \left|
    \int_{\alpha}^{1}\,d\alpha'
    |\xi|
    e^{(\alpha-\alpha')|\xi|}
    |\xi|^{m}
    \int^{\infty}_{-\infty} \,d\eta
    \tilde f(\xi-\eta)
    |\eta|
    \tilde g(\eta)
    e^{\rho_{\theta'}(\alpha)|\xi|}
   \right|^2
   \\&\indeq
   \lec
   \int_{-\infty}^{\infty}
   \,d\xi
   \int_{0}^{1}\,d\alpha
   \\&\indeq\indeq\indeq\indeq\times
   \left|
    \int_{\alpha}^{1}\,d\alpha'
    |\xi|
    e^{(\alpha-\alpha')|\xi|}
   \int^{\infty}_{-\infty}\,d\eta
    |\xi-\eta|^{m}
  |  \tilde f(\xi-\eta) |
    |\eta|
  |  \tilde g(\eta) |
    e^{\rho_{\theta'}(\alpha)|\xi|}
   \right|^2
   e^{\theta'|\xi|/2}
   \\&\indeq\indeq
   +
   \int_{-\infty}^{\infty}
   \,d\xi
   \int_{0}^{1}\,d\alpha
   \\&\indeq\indeq\indeq\indeq\indeq\times
   \left|
    \int_{\alpha}^{1}\,d\alpha'
    |\xi|
    e^{(\alpha-\alpha')|\xi|}
   \int^{\infty}_{-\infty}\,d\eta
  |  \tilde f(\xi-\eta)|
    |\eta|^{m+1}
|    \tilde g(\eta)|
    e^{\rho_{\theta'}(\alpha)|\xi|}
   \right|^2
   e^{\theta'|\xi|/2}
   \fcomma
   \end{split}
   \label{EQ51}
  \end{align}
where $0<\theta'<\theta$.
In addition to the previous agreement on $\tilde f$ and $\tilde g$, we
assume here, in addition, that
$\tilde f=\tilde g=0$ for~$\alpha'\geq1$.
We denote the last two integrals in \eqref{EQ51} by $I_1$ and~$I_2$. We only show how to treat the second term $I_2$ as the first is simpler
since it does not have $m+1$ derivatives in the $x$ variable (observe
the factor of $|\eta|^{m+1}$). For $I_2$, we use
  \begin{equation}
   |\eta|
   e^{\theta'|\eta|/2}
   = |\eta| e^{(\theta-\theta')|\eta|/2}e^{\theta|\eta|/2}
   \lec \frac{1}{\theta-\theta'} e^{\theta|\eta|/2}
   \label{EQ94}
  \end{equation}
and Young's inequality in $\alpha$ to write
\begin{align}
  \begin{split}
   I_2
   &\lec
   \int_{-\infty}^{\infty}\,d\xi
   \int_{0}^{\infty}\,d\alpha'
   \left(
    \int^{\infty}_{-\infty}\,d\eta
    |\tilde f(\xi-\eta,\alpha')|
    e^{\theta'(|\xi-\eta|)/2}
    |\eta|^{m+1}
    |\tilde g(\eta,\alpha')|
    e^{\theta'|\eta|/2}
   \right)^2
   \\&
   \lec
   \frac{1}{(\theta-\theta')^2}
   \int_{-\infty}^{\infty}\,d\xi
   \int_{0}^{\infty}\,d\alpha'
   \\&\indeq\indeq\indeq\indeq\indeq\times
   \left(
    \int^{\infty}_{-\infty}\,d\eta
    |\tilde f(\xi-\eta,\alpha')|
    e^{\theta'(|\xi-\eta|)/2}
    |\eta|^{m}
    |\tilde g(\eta,\alpha')|
    e^{\theta'|\eta|/2}
   \right)^2
   \\&
   \lec_{\theta'}
   \frac{1}{(\theta-\theta')^2}
   \int_{0}^{\infty}\,d\alpha'
   \left(
   \int_{-\infty}^{\infty}\,d\xi
      |\tilde f(\xi,\alpha')| e^{\theta'|\xi|/2}
   \right)^2
   \left(
    \int^{\infty}_{-\infty}
    \,d\xi
    |\xi|^{2m}
    |\tilde g(\xi,\alpha')|^2
    e^{\theta|\xi|}
   \right)
   \\&
   \lec_{\theta'}
   \frac{1}{(\theta-\theta')^2}
   \int_{0}^{\infty}\,d\alpha'
   \int_{-\infty}^{\infty}\,d\xi
   \left(
      |\tilde f(\xi,\alpha')| e^{\theta'|\xi|/2}
   \right)^2
   \left(
    \int^{\infty}_{-\infty}\, d\xi
    |\xi|^{2m}
    |\tilde g(\xi,\alpha')|^2
    e^{\theta|\xi|}
   \right)
   \\&
   \lec
   \frac{1}{(\theta-\theta')^2}
   \int_{0}^{\infty}\,d\alpha'
   \int_{-\infty}^{\infty}\,d\xi
   \left(
      |\xi|^{3}  |\tilde f(\xi,\alpha')|^2 e^{\theta'|\xi|}
   \right)
   \left(
    \int^{\infty}_{-\infty} \, d\xi
    |\xi|^{2m}
    |\tilde g(\xi,\alpha')|^2
    e^{\theta|\xi|}
   \right)
   \\&
   \lec
   C_{\theta} |f|^2_{m,\theta'}
   \frac{
   |g|_{m,\theta}^2
   }
   {
   (\theta-\theta')^2
   }
   \fcomma
  \end{split}
   \llabel{EQ92}
  \end{align}
where in the last step we used H\"older's and the Sobolev inequalities
in~$\alpha'$.

Next, we consider the integral
when $y\in[0,1+i\tan\theta']$,
which we denote by $I$,
while
the integral in $y'$ is over $[1+i\tan\theta',1+\theta']$.
We parameterize the integral~$I$
by  $y=\alpha+i\alpha\tan\theta$
and $y'=\alpha'+i(1+\theta-\alpha')\tan\theta/\theta$ so that
$|dy|=(1+\tan^2\theta)^{1/2}$
and
$|dy'|=(1+\tan^2\theta/\theta^2)^{1/2}$.
We then write
  \begin{align}
  \begin{split}
   I
   &\lec
   \int_{-\infty}^{\infty}
   \,d\xi
   \int_{0}^{1+i\tan\theta'}
   \,dy
   \\&\indeq\indeq\indeq\indeq\indeq\times
   \left|
    \int_{1+i\tan\theta'}^{1+\theta'}
    \,dy'
    |\xi|
    e^{|\xi|e^{|\xi|(y-y')}}
    |\xi|^{m}
    \int^{\infty}_{-\infty}
    \,d\eta
    f(\xi-\eta)
    g(\eta)
    |\eta|
    e^{\rho_{\theta'}(y)}
   \right|^2
  \\&
  \lec_{\theta'}
  \int_{-\infty}^{\infty}\,d\xi
  \int_{0}^{1} \,d\alpha
  \\&\indeq\indeq\indeq\indeq\times
  \left|
   \int_{1}^{1+\theta'}\,d\alpha'
   |\xi|e^{(\alpha-\alpha')|\xi|} |\xi|^{m}
   \int^{\infty}_{-\infty}
   \,d\eta 
   |\tilde f(\xi-\eta,\alpha')| \, |\tilde g(\eta,\alpha')|
   |\eta| e^{\theta'|\xi|/2}
  \right|^2
  \\&
  \lec  
  \int_{-\infty}^{\infty}\,d\xi
  \int_{0}^{1} \,d\alpha
  \\&\indeq\indeq\indeq\indeq\times
  \left|
   \int_{1}^{1+\theta'}\,d\alpha'
   |\xi|e^{(\alpha-\alpha')|\xi|} 
   \int^{\infty}_{-\infty}
   \,d\eta
    |\xi-\eta|^{m} |\tilde f(\xi-\eta,\alpha')|\,  |\tilde g(\eta,\alpha')|
   |\eta| e^{\theta'|\xi|/2}
  \right|^2
  \\&\indeq\indeq
  +
  \int_{-\infty}^{\infty}\,d\xi
  \int_{0}^{1} \,d\alpha
  \\&\indeq\indeq\indeq\indeq\indeq\times
  \left|
   \int_{1}^{1+\theta'}\,d\alpha'
   |\xi|e^{(\alpha-\alpha')|\xi|} 
   \int^{\infty}_{-\infty} 
   \,d\eta  
   | \tilde f(\xi-\eta,\alpha')| |\eta|^{m+1} |\tilde g(\eta,\alpha')|
   e^{\theta'|\xi|/2}
  \right|^2
  \fperiod
  \end{split}
   \label{EQ93}
  \end{align}
\colb
As above, we denote the last two integrals by $I_1$ and $I_2$ and only
consider $I_2$, as the other one is easier.
Since
  \begin{align}
  \begin{split}
   \frac{\theta'}{2}|\xi|
   + \frac{\alpha-\alpha'}{2}|\xi|
   &\leq
   \frac{1+\theta'-\alpha'}{2}|\xi|
   +
   \frac{\theta'}{2}|\xi|
   \leq
   \frac{1+\theta'-\alpha'}{2}|\eta|
   +
   \frac{1+\theta'-\alpha'}{2}|\xi-\eta|
   +
   \frac{\theta'}{2}|\xi|
   \\&
   \leq
   \frac{1+\theta-\alpha'}{2}|\eta|
   +
   \frac{1+\theta'-\alpha'}{2}|\xi-\eta|
   +
   \frac{\theta-\theta'}{2}|\eta|
   \fcomma
  \end{split}
   \label{EQ97}
  \end{align}
we get
  \begin{align}
  \begin{split}
   I_2
    &\lec
  \int_{-\infty}^{\infty}\,d\xi
  \int_{0}^{1} \,d\alpha
  \biggl|
   \int_{1}^{1+\theta'}
   \,d\alpha'
   |\xi|e^{(\alpha-\alpha')|\xi|/2}
   \int^{\infty}_{-\infty}
   \,d\eta
  \\&\indeq\indeq\indeq\indeq\indeq\times
   e^{(1+\theta'-\alpha')|\eta|/2}
   |\tilde f(\xi-\eta,\alpha')| |\eta|^{m+1}
      e^{(1+\theta-\alpha')|\eta|/2}
      |\tilde g(\eta,\alpha')|
      e^{(\theta-\theta')|\eta|/2}
  \biggr|^2
  \\&
  \lec
  \frac{1}{(\theta-\theta')^2}
  \int_{-\infty}^{\infty}\,d\xi
  \int_{0}^{1} \,d\alpha
  \biggl|
   \int_{1}^{1+\theta'}
   \,d\alpha'
   |\xi|e^{(\alpha-\alpha')|\xi|/2}
   \int^{\infty}_{-\infty} 
   \,d\eta
  \\&\indeq\indeq\indeq\indeq\indeq\indeq\times
   e^{(1+\theta'-\alpha')|\xi-\eta|/2}
   |\tilde f(\xi-\eta,\alpha')| \, |\eta|^{m}
      e^{(1+\theta-\alpha')|\eta|/2}
      |\tilde g(\eta,\alpha')|
  \biggr|^2
  \\&
  \lec
  \frac{1}{(\theta-\theta')^2}
  \int_{-\infty}^{\infty}\,d\xi
  \int_{-\infty}^{\infty} \,d\alpha
   \biggl|
   \int^{\infty}_{-\infty} 
   \,d\eta
  \\&\indeq\indeq\indeq\indeq\indeq\times
   e^{(1+\theta'-\alpha')|\xi-\eta|/2}
   |\tilde f(\xi-\eta,\alpha') | \, |\eta|^{m}
      e^{(1+\theta-\alpha')|\eta|/2} |\tilde g(\eta,\alpha')|
  \biggr|^2
  \\&
  \lec
  \frac{1}{(\theta-\theta')^2}
  \int_{-\infty}^{\infty} \,d\alpha
   \left|
   \int^{\infty}_{-\infty} 
   \,d\xi
   e^{(1+\theta'-\alpha')|\xi|/2}
  | \tilde f(\xi,\alpha') |
   \right|^2
  \\&\indeq\indeq\indeq\indeq\indeq\times
   \left(
    \int^{\infty}_{-\infty} 
    \,d\xi
     e^{(1+\theta-\alpha')|\xi|}|\xi|^{2m} |\tilde g(\eta,\alpha')|^2
  \right)
  \\&
  \lec
  \frac{1}{(\theta-\theta')^2}
  \int_{-\infty}^{\infty} \,d\alpha
   \left|
   \int^{\infty}_{-\infty} 
   \,d\xi
   |\xi|^{3}
   e^{(1+\theta'-\alpha')|\xi|}
  | \tilde f(\xi,\alpha')|^2
   \right|^2
  \\&\indeq\indeq\indeq\indeq\indeq\times
   \left(
    \int^{\infty}_{-\infty} 
    \,d\xi
     e^{(1+\theta-\alpha')|\xi|}|\xi|^{2m} |\tilde g(\eta,\alpha')|^2
  \right)
  \\&
  \lec
   C_{\theta} |f|^2_{m,\theta'}
   \frac{
   |g|_{m,\theta}^2
   }
   {
   (\theta-\theta')^2
   }
   \fcomma
  \end{split}
  \label{EQ95}
  \end{align}
where in the last step we used H\"older's and the Sobolev inequalities.

The last integral, when $y'\in [1+\theta',\infty[$, is the simplest and we omit the details.
This shows how to estimate
the analytic norm of~$f\partial_{x}g$.

Now, let $\bu$ and $\bv$ be as in the statement.
In the rest of the proof, we sketch the argument needed to settle the
general case using the above argument.
According to \eqref{EQ11}, we need to bound
the expression
  \begin{align}
  \begin{split}
  |\mathbb{P}(\bv\cdot \nabla \bu)|^\text{D}_{m,\theta'}
  &=
  \sum_{i+j\leq m} |\pardx^i\pardy^j\mathbb{P}(\bv\cdot \nabla \bu)|_{\theta'}
  \lec
  \sum_{i+j\leq m} |\mathbb{P}(\pardx^i\pardy^j(\bv\cdot \nabla \bu))|_{\theta'}    
  \fperiod
  \end{split}
  \label{EQ98}
  \end{align}
In the last inequality, we used~\eqref{prop_projection}.
Now, we take a closer look at the last expression in~\eqref{EQ98}.
When $i\ge1$
or if $i+j\leq m-1$,
then the terms can be treated same as above, so we only
need
to address the case $(i,j)=(0,m)$,
when we can apply the Leibniz rule.
Most of the terms lead to no
derivative loss
and can be treated directly, so we only need to bound
  \begin{align}
   |\mathbb{P}(\bu\cdot\nabla \partial_{y}^{m}\bv)|^\text{D}_{m,\theta'}
   \fperiod
   \label{EQ99}
   \end{align}
Note that
  \begin{align}
  \begin{split}
   &
   \bu\cdot\nabla \partial_{y}^{m}\bv
   =
  \begin{pmatrix}
    u_1 \partial_{x}\partial_{y}^{m} u_1 &   u_2 \partial_{y}^{m+1} u_1 \\
     u_1 \partial_{x}\partial_{y}^{m} u_2 &  u_2 \partial_{y}^{m+1} u_2 \\
  \end{pmatrix}
  =
  \begin{pmatrix}
    u_1 \partial_{x}\partial_{y}^{m} u_1 &   u_2 \partial_{y}^{m+1} u_1 \\
    u_1 \partial_{x}\partial_{y}^{m} u_2 & - u_2 \partial_{x}\partial_{y}^{m} u_1 \\
  \end{pmatrix}
  \fperiod
  \end{split}
  \label{EQ100}
  \end{align}
Note that three of the entries have at least one derivative in the $x$ variable
and can thus be treated the same way as above. Thus we only need to address the
term~$u_2 \partial_{y}^{m+1} u_1$.
This term can be treated by integrating by parts in the $y$ variable once, using first
 that $u_2$ vanishes on the boundary and then also applying
the incompressibility condition. 
This shows that the structure of the projection operator transfers the $y$ derivative to the $x$
derivative and, therefore, the term $\mathbb{P}\bu\cdot\bnabla \bv $ can be treated in the same fashion as $f\pardx g$. 
This concludes the proof of  Proposition~\ref{P01}. 
\colb

\subsection{The analytic estimate in the linear case}
Here we prove Lemma~\ref{analy_estimate}.
Given the expression for $F_\ep$ in~\eqref{transport.reg.operator.form} and the estimate~\eqref{EQ11}, for $0<\theta<\theta'$, one may write
\begin{equation}
|\Freg(\bv, \bureg,t)|^{\text{D}}_{m,\theta'}\leq \cpCau R\frac{1}{\theta'-\theta}\left(|\bureg|^\text{D}_{m,\theta'}+\|\bureg\|_{m,a}\right)\fperiod 
\llabel{EQ151}\end{equation}
Therefore, when $0<\theta<\theta(s)$, we have 
\begin{align} 
&
\left|\int_0^t \Freg(\bv, \bureg,s) ds\right|^{\text{D}}_{m,\theta} 
 \leq  \cpCau R \int_0^t \frac{|\bureg|^\text{D}_{m,\theta(s)}+\|\bureg\|_{m,a}}{\theta(s)-\theta}ds \nonumber \\&\indeq
\leq 	\cpCau R \int_0^t \left(1-\beta s /(\theta_0-\theta(s))\right)^{-\gamma} \left(1-\beta s /(\theta_0-\theta(s))\right)^{\gamma} \frac{|\bureg|^\text{D}_{m,\theta(s)}+\|\bureg\|_{m,a}}{\theta(s)-\theta}ds 
\nonumber \\&\indeq\leq \cpCau R  \bigl( |\bureg|^{(\gamma)}_{m,\theta,\beta}    +\sup_{t\in[0,T]} \| \bureg\|_{m,a}\bigr) 
 \int_0^t \left(1-\beta s /(\theta_0-\theta(s))\right)^{-\gamma} \frac{1}{\theta(s)-\theta} ds \nonumber 
  \fperiod
 \end{align}
We now choose $\theta(s)$ in the following way. Let
$\rho=\theta_0-\theta-\beta s$ with $\theta(s)=\theta +\rho/2$, which
then implies
$\theta_0-\theta(s)-\beta s=(\theta_0-\theta-\beta s)/2$. Using this in the above estimate, we may write
 \begin{align}
   \begin{split}
   &
 	\left|\int_0^t \Freg(\bv, \bureg,s) ds\right|^{\text{D}}_{m,\theta} 
   \\&\indeq
 	\leq \cpCau R  \bigl( |\bureg|^{(\gamma)}_{m,\theta,\beta}    +\sup_{t\in[0,T]} \| \bureg\|_{m,a}\bigr) 2^{\gamma +1 } \int_0^t\frac{(\theta_0-\theta(s))^\gamma}{(\theta_0-\theta-\beta s)^{\gamma+1}}ds 
\\&\indeq
 	\leq   	 \cpCau R  \bigl( |\bureg|^{(\gamma)}_{m,\theta,\beta}    +\sup_{t\in[0,T]} \| \bureg\|_{m,a}\bigr) 2^{\gamma +1 }  
 	\frac{(\theta_0-\theta)^\gamma}{\beta \gamma}(\theta_0-\theta-\beta t)^{-\gamma}
  \fperiod
  \end{split}
 \label{DC}
\end{align}
 The last estimate immediately leads to \eqref{est:Freg_unif} in Lemma~\ref{analy_estimate}, while
the inequality \eqref{est:bureg_unif} follows from
\eqref{transport.reg.operator.form} and~\eqref{est:Freg_unif}.

 \subsection{The Sobolev estimate in the linear case}
 \label{subsect:proof_Lemma5.9}
Here we prove Lemma~\ref{Sobol_estimate}.
  Using the standard energy arguments, the expression $\|\bu\|_{m,a}$ satisfies 
  \begin{equation}
  	\frac{1}{2}\frac{d}{dt} \|\bu\|^2_{m,a}\leq \cSob \| \bv\|_{m,a} \|\bu\|^2_{m,a} +\BT_1+\BT_2  \label{est::energy}
  	\fcomma
  	\end{equation}
where $\BT_i$ denotes terms deriving from the fact that the $\|\cdot\|_{m,a}$ involves integration in the normal variable $y$ on~$[a,\infty[$. 
More precisely,
 $$
 \BT_1=\frac{1}{2}\left| \int_{\mathbb{R}} \left(\bv\cdot\bn \left(D^m \bu\right)^2\right)_{y=a} dx\right| \andand
 \BT_2=\sum_{l}\int_{\mathbb{R}} \left(D^l p \,\bn\cdot D^l \bu \right)_{y=a} dx
 \fperiod
 $$
First, we treat~$\BT_1$, when we write
\begin{align}
  \BT_1
  \leq \frac{1}{2}R  \int_{\mathbb{R}} \left( D^m \bu\right)^2_{y=a} dx \leq  R \cSobCau \left(\frac{|\bu|^\text{D}_{m,\theta(t)} }{\theta(t)}\right)^2
 \fcomma
 \label{est::BT1}
 \end{align}
where $\theta(t)$ is such that $\theta(t)>\bar{\theta}$ for $ t\in [0,\theta_0/\beta]$.  
We have used the following estimate on $\left( D^m \bu\right)^2_{y=a}$:
  \begin{align}
    \begin{split}
 \int_{\mathbb{R}} \left( D^m \bu\right)^2_{y=a} dx         &\leq  \int_{\mathbb{R}} \left|D^m \bu(x,\cdot)\right|^2_{L^\infty(a-\bar{\theta},a+\bar{\theta})}dx  \\
    	& \leq \cSob \int_{\mathbb{R}} \left(\left\|D^m \bu(x,\cdot)\right\|^2_{L^2(a-\bar{\theta},a+\bar{\theta})}+
    \left\|\pardy D^m \bu(x,\cdot)\right\|^2_{L^2(a-\bar{\theta},a+\bar{\theta})}	\right)dx  \\
    	& \leq   \cSob\left(  \left( |\bu|^\text{D}_{m,\theta(t)} \right)^2 +\cCau \left(\frac{|\bu|^\text{D}_{m,\theta(t)} }{\theta(t)}\right)^2 \right)  
    	\leq   \cSobCau \left(\frac{|\bu|^\text{D}_{m,\theta(t)} }{\theta(t)}\right)^2
     \fperiod
  \end{split}
     \label{est:Dmu_at_a}
    \end{align}
Next, for~$\BT_2$, we have
  		\begin{align}
\BT_2				&\leq \sum_l\left(\int _{\mathbb{R}} \left(D^l p \right)^2_{y=a}  dx\right)^{1/2} \left(\int _{\mathbb{R}} \left(D^l \bu \right)^2_{y=a}  dx\right)^{1/2} \nonumber \\
				& \leq \cSobCau \frac{|\bu|^\text{D}_{m,\theta(t)} }{\theta(t)} R \left(|\bu|^\text{D}_{m,\theta(t)}+ \|\bu\|_{m,a} \right)
				\leq \cSobCau R\left(\|\bu\|_{m,a}^2+\left(\frac{|\bu|^\text{D}_{m,\theta(t)}}{\theta(t)}\right)^2 \right)
\fcomma
\label{est::BT2}
  		\end{align}
where in the second inequality, we have used the inequality~\eqref{est:Dmu_at_a} and the estimate on the pressure
    \begin{align}
      &
  	\left(	\int_{\mathbb{R}} \left( D^m p\right)^2_{y=a} dx   \right)^{1/2}      
  	\leq  \left(\int_{\mathbb{R}} \left|D^m p(x,\cdot)\right|^2_{L^\infty(a-\bar{\theta},a+\bar{\theta})}dx  \right)^{1/2}  \nonumber \\&\indeq
  		 \leq \cSob \left(\int_{\mathbb{R}} \left(\left\|D^m p(x,\cdot)\right\|^2_{L^2(a-\bar{\theta},a+\bar{\theta})}+
  		\left\|\pardy D^m p(x,\cdot)\right\|^2_{L^2(a-\bar{\theta},a+\bar{\theta})}	\right)dx  \right)^{1/2}   \nonumber \\&\indeq
  		\leq  \cSob\left\| p\right\|_{H^{m+1}(a-\bar{\theta},a+\bar{\theta})} 
  		\leq  \cSob \left( \|\bv\|_{H^{m}(a-\bar{\theta},a+\bar{\theta})}  \|\bu\|_{H^{m}(a-\bar{\theta},a+\bar{\theta})} +\|p\|_{L^2}\right) \nonumber \\&\indeq
  		\leq  \cSob \left(R|\bu|^\text{D}_{m,\theta(t)}+ \|\bv\|_{H^2} \|\bu\|_{H^2}\right)
		\nonumber \\&\indeq
  		\leq  \cSob  \left(R|\bu|^\text{D}_{m,\theta(t)}+ R \left(|\bu|^\text{D}_{m,\theta(t)}+\|\bu\|_{m,a}\right)\right)
  		= \cSob R \left(|\bu|^\text{D}_{m,\theta(t)}+ \|\bu\|_{m,a} \right)  
\fperiod
\label{est::pressure}
  	\end{align}
Using \eqref{est::BT1} and \eqref{est::BT2} in \eqref{est::energy}, we get the differential inequality
  \begin{equation}
  \frac{1}{2}\frac{d}{dt} \|\bu\|^2_{m,a}\leq \cSobCau R \left(\|\bu\|^2_{m,a}+ \left(\frac{|\bu|^\text{D}_{m,\theta(t)} }{\theta(t)}\right)^2\right)
  \fperiod
  \llabel{EQ152}
  \end{equation}
  The above estimate for $t<(\theta_0-\bar{\theta})/\beta$ in order to have a sufficient analyticity to estimate, using the Cauchy inequality, 
  the higher derivatives $D^{m+1}\bu$ at $y=a$. 
  
Using the Gronwall Lemma, for $ 0<t<(\theta_0 -\bar{\theta})/\beta$, we may write
  \begin{align}
    \begin{split}
  &
  \|\bu(\cdot, t)\|^2_{m,a}\leq e^{2\cSobCau Rt}\left(\|\buin\|^2_{m,a}+  2\cSobCau \int_0^t  |\bv|^\text{D}_{m,\theta(s)}    \left(  \frac{|\bu|^{\text{D}}_{m,\theta(s)}}{\theta(s)} \right)^2 ds \right)    \\
  &\indeq\leq e^{2\cSobCau Rt}\left(\|\buin\|^2_{m,a}+  2\cSobCau R \int_0^t    \left(  \frac{|\bu|^{\text{D}}_{m,\theta(s)}}{\theta(s)} \right)^2 ds \right)    \\
  &\indeq\leq e^{2\cSobCau Rt}\left(\|\buin\|^2_{m,a}+  2\cSobCau R \left( |\bu|^{(\gamma)}_{m,\theta_0,\beta}\right)^2   \int_0^t    \left(1-\frac{\beta s}{\theta_0-\theta(s)}\right)^{-2\gamma}   
  \left(  \frac{1}{\theta(s)} \right)^2 ds \right)    \nonumber \\
 &\indeq
 \leq e^{2\cSobCau Rt}\|\buin\|^2_{m,a} + e^{2\cSobCau Rt} 2\cSobCau R 2^{2(\gamma+1)} \frac{\theta_0^{2\gamma}}{\bar{\theta}^{2(\gamma+1)}}\frac{\theta_0-\bar{\theta}}{\beta}\left( |\bu|^{(\gamma)}_{m,\theta_0,\beta}\right)^2\fperiod
   \llabel{EQ155} 
  \end{split}
  \label{Sobol.with.diamond}
  \end{align}
  To get the last inequality, we have chosen $\theta(s)=\left(\theta_0-\beta s\right)/2$ and used that $\theta(s)\geq \bar{\theta}/2$ and $t\leq \left(\theta_0-\bar{\theta}\right)/\beta$.  
From the above estimate, we conclude that 
\begin{equation}
 \|\bu(\cdot, t)\|_{m,a}\leq e^{\cSobCau Rt}\|\buin\|_{m,a}+ e^{\cSobCau Rt}\sqrt{2\cSobCau R}2^{\gamma+1}\frac{\theta_0^{\gamma+1/2}}{\bar{\theta}^{\gamma+1}}\frac{1}{\sqrt{\beta}}  |\bu|^{(\gamma)}_{m,\theta_0,\beta}\fperiod 
\llabel{EQ153}\end{equation}
 Therefore, we have obtained the estimate~\eqref{est:Sob_bureg} with 
 \begin{equation}
 D(R,T) =e^{\cSobCau Rt}\sqrt{2\cSobCau R}2^{\gamma+1}\frac{\theta_0^{\gamma+1/2}}{\bar{\theta}^{\gamma+1}}\frac{1}{\sqrt{\beta}}\fcomma
 \llabel{EQ154}\end{equation} 
which concludes the proof of Lemma~\ref{Sobol_estimate}.

  \section{The analytic-Sobolev estimate in the nonlinear case}
  \label{sect:proof_p6.2_l6.3_6.4}
  
  \subsection{Proof of Proposition~\ref{prop:contractiveness}} 
We first introduce constants
  \be
A_1=\cpCau\frac{2^{\gamma+1}}{\gamma},  \qquad A_2=\cpro\frac{2^\gamma \theta_0}{(1+\gamma)},
\qquad A_3=\cSobCau e^{\cSobCau R T}, \qquad A_4= \frac{\cSobCau}{\bar{\theta}^{\gamma+1}}e^{\cSobCau R T}
\fperiod\label{express_ABC}
  \ee
The estimates given in Lemma~\ref{differences.analytic.est} and \ref{differences.Sob.est} are equivalent to 
  \begin{align}
&
	|\bznpo|^{(\gamma)}_{m-1,\theta_0, \beta_n,  T }
\\&\indeq
\leq \left(1-\frac{R}{\beta_n}A_1\right)^{-1} \frac{R}{\beta_n} \left(A_1 \sup_{0\leq t<T} \| \bznpo \|_{m-1,a} + 
	A_2  |||\bzn |||^{(\gamma)}_{m-1,\theta_0,\beta_n,a,T}   \right)
	\label{eq::127}
	\end{align}
and
\begin{align}
 \sup_{t}\|\bznpo\|_{m-1,a} & \leq 
 \sqrt{\frac{R}{\beta_n}}  \left( A_3   \sup_{0\leq t<T} \|    \bzn \|_{m-1,a} + A_4    |\bznpo|^{(\gamma)}_{m-1,\theta_0,\beta_n,   T } \right) \fcomma
 \label{eq::128} 
  \end{align}
provided that $\beta_n$ is large enough to ensure  $\left(1-RA_1/{\beta_n}\right)>0$; to this end, recalling that $\beta_n\geq \beta/2$, we  impose   
  \be
  \beta>4RA_1
  \label{first.cond.onbeta}
  \ee
so that $\left(1-{R}A_1/{\beta_n}\right)^{-1}<2$, for all~$n$.  Note that if $\beta$ satisfies \eqref{first.cond.onbeta}, then
  $$
  e^{\cSobCau R T}\leq \exp{\frac{\cSobCau \theta_0 \gamma}{2^{\gamma+3} \cpCau}}
  \fperiod
  $$
Inserting the estimate \eqref{eq::127} into \eqref{eq::128}, we get
  \begin{align}
     \begin{split}
  &
  \biggl(1-A_1A_4\left(\frac{R}{\beta_n}\right)^{3/2}\left(1-{R}A/{\beta_n}\right)^{-1}\biggr)
  \sup_{t}\|\bznpo\|_{m-1,a}
  \\&\indeq
  \leq \sqrt{\frac{R}{\beta_n}}\left(A_3 +A_2A_4\frac{R}{\beta_n}\left(1-{R}A/{\beta_n}\right)^{-1} \right)
   |||\bzn |||^{(\gamma)}_{m-1,\theta_0,\beta_n,a,T}
   \fperiod
   \end{split}
  \end{align}
Using \eqref{first.cond.onbeta} and imposing also
  \be
  \beta >2\left(4A_1A_4\right)^{2/3} R
  \fcomma\label{second.cond.onbeta}
  \ee
  so that $  \beta_n >\left(4A_1A_4\right)^{2/3} R$, 
  one readily derives 
  $$
   \sup_{t}\|\bznpo\|_{m-1,a} \leq 2 \sqrt{\frac{R}{\beta_n}} \left(A_3+A_2A_4\right)  |||\bzn |||^{(\gamma)}_{m-1,\theta_0,\beta_n,a,T}
   \fperiod
  $$
Inserting the above bound in \eqref{eq::127}, we get
  $$
  	|\bznpo|^{(\gamma)}_{m-1,\theta_0, \beta_n,  T }\leq   2  \frac{R}{\beta_n}\left(2 \sqrt{\frac{R}{\beta_n}}A_1\left(A_3+A_2A_4\right)+A_2  	\right)  |||\bzn |||^{(\gamma)}_{m-1,\theta_0,\beta_n,a,T}
	\fperiod
  $$
 Adding the two estimates above and using $R/\beta_n<1/2$, we may finally write
  $$
   |||\bznpo |||^{(\gamma)}_{m-1,\theta_0,\beta_n,a,T} \leq A \sqrt{\frac{R}{\beta_n}} |||\bzn |||^{(\gamma)}_{m-1,\theta_0,\beta_n,a,T}, 
  $$
  where 
  \be
  A= (A_3+A_2A_4)(2+2A_1+\sqrt{2}) \label{express_of_A}
 \fperiod
  \ee
  The constants $A_i$ are given in~\eqref{express_ABC}.
  
  Exactly with the same procedure, one can prove the estimate in the norm\\$|||\cdot |||^{(\gamma)}_{m-1, \theta_0,\beta, a,T}$.

  \colb
  \subsection{Proof of Lemma~\ref{differences.analytic.est}}
The proof of Lemma~\ref{differences.analytic.est} goes along the same lines as the proof of Lemma~\ref{analy_estimate}. 
In fact, from \eqref{diff_operatorform}, we may write
  \begin{align}
    \begin{split}
  |\bznpo|^\text{D}_{m-1,\theta}\leq &\int_0^t  |\mathbb{P}\bu^n\cdot \bnabla \bznpo|^\text{D}_{m-1,\theta} ds +  \int_0^t  |\mathbb{P}\bzn\cdot \bnabla \bu^n|^\text{D}_{m-1,\theta} ds 
   \\  
  \leq & \cpCau R \int_0^t \frac{|\bznpo|^\text{D}_{m-1,\theta(s)}+\|\bznpo\|_{m-1,a}}{\theta(s)-\theta}ds
   \\&\indeq
  +
 \cpro R\int_0^t  \left(|\bzn|^\text{D}_{m-1,\theta(s)} + \|\bzn\|_{m-1,a}  \right)ds \\
=&  \cpCau R \int_0^t \frac{\left(|\bznpo|^\text{D}_{m-1,\theta(s)}+\|\bznpo\|_{m-1,a}\right) \left(\theta_0-\theta(s)-\beta_n s\right)^{\gamma}}
{\left(\theta(s)-\theta\right) \left(\theta_0-\theta(s)-\beta_n s\right)^{\gamma}}ds
 \\&\indeq+
 \cpro R\int_0^t  \frac{\left(|\bzn|^\text{D}_{m-1,\theta(s)} + \|\bzn\|_{m-1,a}  \right)   \left(\theta_0-\theta(s)-\beta_n s\right)^{\gamma}}
{ \left(\theta_0-\theta(s)-\beta_n s\right)^{\gamma}}ds \\
 \leq &R \cpCau  \frac{2^{\gamma+1}\theta_0^\gamma}{\gamma \beta_n} \left(|\bznpo|^{(\gamma)}_{m-1,\theta_0,\beta_n,T} +\sup_{0\leq t \leq T} \|\bznpo\|_{m-1,a} \right)
  \left(\theta_0-\theta-\beta_n t \right)^{-\gamma}\\&\indeq+
    R \cpro \frac{2^\gamma\theta_0^{1-\gamma}}{(1-\gamma)\beta_n} \left(|\bzn|^{(\gamma)}_{m-1,\theta_0,\beta_n,T} +\sup_{0\leq t \leq T} \|\bzn\|_{m-1,a} \right)
  \fperiod
  \end{split}
  \label{HG}
  \end{align}
To get the last inequality, we have chosen $\theta(s)$ so that $\theta_0-\theta(s)-\beta_n s=(\theta_0-\theta-\beta_n s)/2$, as we did proving the 
  Lemma~\ref{analy_estimate} to get the estimate~\eqref{DC}.
  Multiplying both sides of \eqref{HG} times  $\left(\theta_0-\theta-\beta_n t \right)^{\gamma} $, one gets
  \begin{align}
  	\begin{split}
  |\bznpo|^{(\gamma)}_{m-1,\theta_0,\beta_n,T} \leq &
  R \cpCau \frac{2^{\gamma+1}}{\gamma \beta_n} 
  \left(	|\bznpo|^{(\gamma)}_{m-1,\theta_0, \beta_n,T } +   \sup_{0\leq t<T}  \| \bznpo\|_{m-1,a} \right) \\&\indeq+
 R \cpro \frac{2^\gamma\theta_0}{(1-\gamma)\beta_n} \left( |\bzn|^{(\gamma)}_{m-1,\theta_0, \beta_n,T} +   \sup_{0\leq t<T}  \| \bzn\|_{m-1,a}\right)
\fperiod
   \llabel{EQ04}
\end{split}
  \end{align}

  \subsection{Proof of Lemma~\ref{differences.Sob.est}}
Applying $D^l$ to \eqref{Diff}, taking the product with $D^l\bznpo$, integrating on $\mathbb{R}\times [a,\infty]$, 
and taking the sum for $l=0,\ldots,m-1$, we obtain
  \begin{align}
	\frac{1}{2}\frac{d}{dt}\|\bznpo\|^2_{m-1,a} &\leq 
	\cSob \left(  \|\bu^n\|_{m-1,a}   \|\bznpo\|^2_{m-1,a}+ 
	\|\bu^n\|_{m,a}  \|\bznpo\|_{m-1,a}  \|\bzn\|_{m-1,a}\right)    \nonumber \\
	 &\indeq+\BT_1+\BT_2 \comma 0<t<(\theta_0 -\bar{\theta})/\beta\fcomma  \label{348}
\end{align}
where 
\begin{equation}
  \begin{split}
\BT_1&=\frac{1}{2}\left| \int_{\mathbb{R}} \left(\bzn\cdot\bn \left(D^{m-1} \bznpo\right)^2\right)_{y=a} dx\right|
  \\
\BT_2&=\int_{\mathbb{R}} \left(D^{m-1} \pi^{n+1} \,\bn\cdot D^{m-1} \bznpo \right)_{y=a} dx
\fperiod
  \end{split}
   \llabel{EQ05}
\end{equation}
The expressions of the boundary terms $\BT_1$ and $\BT_2$ above are very similar to the expressions of the boundary terms appearing in Section~\ref{subsect:proof_Lemma5.9}. 
Exactly in the same fashion, one can derive estimates similar to
\eqref{est::BT1} and \eqref{est::BT2}, which read
\begin{align}
	\BT_1&  \leq  R \cSobCau \left(\frac{|\bznpo|^\text{D}_{m-1,\theta(t)} }{\theta(t)}\right)^2\fcomma \nonumber \\
	\BT_2 & \leq \cSobCau R\left(\|\bznpo\|_{m-1,a}^2+\left(\frac{|\bznpo|^\text{D}_{m-1,\theta(t)}}{\theta(t)}\right)^2 \right)\fperiod \nonumber 
\end{align}
Using these bounds in \eqref{348}, we readily obtain
   \begin{align}
 	\frac{1}{2}\frac{d}{dt}\|\bznpo\|^2_{m-1,a} \leq & 
 	\cSobCau R \left(    \|\bznpo\|^2_{m-1,a}+ 
 	   \|\bzn\|^2_{m-1,a} +
 	 \frac{|\bznpo|^{\text{D}}_{m-1,\theta(t)}}{\theta(t)}\right) ^2
 	\fcomma
   \llabel{EQ06}
 \end{align}
for $0<t<(\theta_0 -\bar{\theta})/\beta$.
Applying Gronwall's lemma, we get for $0<t<(\theta_0 -\bar{\theta})/\beta$,
 \begin{align}
 &
 \| \bznpo \|^2_{m-1,a}
  \\&\indeq
  \leq e^{2\cSobCau Rt}
 \left( 2\cSobCau tR\sup_{0\leq t\leq (\theta_0-\bar{\theta})/\beta} \|\bzn\|^2_{m-1,a} + 2\cSobCau R\int_0^t \left( \frac{|\bznpo|^\text{D}_{m-1,\theta(s)}}{\theta(s)}\right)^2ds
\right)
\nonumber \\
&\indeq\leq e^{2\cSobCau Rt} \Biggl( 2\cSobCau R\frac{\theta_0}{\beta}\sup_{0\leq t\leq (\theta_0-\bar{\theta})/\beta} \|\bzn\|^2_{m-1,a}
\\&\indeq\indeq\indeq\indeq\indeq\indeq\indeq\indeq
+ 2\cSobCau R 2^{2} \frac{\theta_0^{2\gamma}}{\bar{\theta}^{2(\gamma+1)}}
\frac{\theta_0-\bar{\theta}}{\beta}\left( |\bznpo|^{(\gamma)}_{m-1,\theta_0,\beta,T} \right)^2
\Biggr) \fcomma\nonumber
 \end{align} 
from which it follows that
 \be
  \| \bznpo \|_{m-1,a} \leq  e^{\cSobCau Rt} \frac{\sqrt{R}}{\sqrt{\beta}} \cSobCau \left(   \sup_{0\leq t\leq (\theta_0-\bar{\theta})/\beta} \|\bzn\|_{m-1,a} + \frac{|\bznpo|^{(\gamma)}_{m-1,\theta_0,\beta,T}}{\bar{\theta}^{\gamma+1}} \right)
  \fcomma
 \ee
for
$0<t<(\theta_0 -\bar{\theta})/\beta$.

  \section{Conclusions}
  
  In this paper, using the velocity formulation of the Euler equations in the half-space, we have shown that a datum with analytic-Sobolev regularity (analytic close to the boundary and Sobolev-regular away from the boundary) gives rise to a solution with the same regularity. 
We believe that the domain of analyticity of our solution is optimal: we do not need the initial datum to be analytic away from the boundary while, close to the boundary, we do not need to have uniform analyticity; moreover, close to the boundary, we can have a domain of analyticity whose radius of analyticity can have an arbitrary linear dependence from the distance from the boundary; in other words, the angle $\theta$ in Fig.~\ref{fig1.a} needs to be less than $\pi/2$, improving the limitation $\theta<\pi/4$ of the previously appeared results on analytic solutions of the Euler equations.

The main technical tools that we used to get this result are: 
First, we have established new estimates on the Leray projection operator, namely we obtain obtained that it is bounded in the analytic-Sobolev norm and supports the Cauchy estimate for analytic functions. Here the main challenge is the nonlocal character of the projection operator. 
Second, we have shown the existence of the solution for the linearized problem by an analytic regularization of the equation and proving uniform (with respect to the size of the regularization) estimates of the obtained solution. 
Third, we have constructed the solution through an iterative procedure that, given that the $(n+1)$-th iterate is defined implicitly in terms of the $n$-th iterate, is nonstandard in the context of the abstract Cauchy-Kowalevskaya theorem. We have achieved the convergence of this procedure using a weighted norm and shrinking at each iterative step the domain of analyticity, an idea inspired by~\cite{Asa88}. 

Several problems would be interesting to address. If one imposes to the Navier-Stokes equations an initial datum 
with the analytic-Sobolev regularity that we have investigated in this paper, can one construct the NS solution with the same analytic-Sobolev regularity and prove the 
convergence to the corresponding Euler solution? Of interest also is the case when the Euler datum does not satisfy the no-slip condition, 
as in \cite{ACS24}, where the incompatibility between the initial and the boundary conditions gives rise to singular terms that are challenging to control. 

The geometry we have considered in this paper is the half-plane or the half-space; it would be interesting to consider the case when the boundary is curved, 
or is an interface that moves according to the dynamics of the problem, as is the case with vortex sheets or vortex layers \cite{CLS20} or vortex patches \cite{LSZ2020}
 or the hydrostatic equations \cite{KTVZ11}. 

\section*{Acknowledgments}
This work is based upon work supported by the National Science Foundation under Grant No.~DMS-1928930 while the authors were in residence at the Simons Laufer Mathematical Sciences Institute (formerly MSRI) in Berkeley, California, during the summer of~2023.
IK was supported in part by the NSF grant DMS-2205493,
MCL was supported in part by the grant PRIN~2022M9BKBC, 
while
MS was supported in part by the grant PRIN-PNRR P202254HT8.

\colb


\begin{thebibliography}{10}
\small	
	\bibitem{AM1984}
	S.~Alinhac and G.~M\'etivier.
	\newblock Propagation de l'analyticit\'e des solutions de syst\`emes
	hyperboliques non-lin\'eaires.
	\newblock {\em Inventiones Mathematicae}, 75(2):189 – 204, 1984.
	
	\bibitem{AM1986}
	S.~Alinhac and G.~M\'etivier.
	\newblock Propagation de l'analyticit\'e locale pour les solutions de
	l'\'equation d'{E}uler.
	\newblock {\em Archive for Rational Mechanics and Analysis}, 92(4):287 – 296,
	1986.
	
	\bibitem{ACS24}
	A.~Argenziano, M.~Cannone and M.~Sammartino. 
	\newblock {N}avier–{S}tokes equations in the half space with non compatible data. 
	\newblock{\em J.~Math. Fluid Mech.}, 26(2): n.32, 2024
	
	\bibitem{Asa88}
	K.~Asano.
	\newblock A Note on the Abstract Cauchy-Kowalewski Theorem.
	\newblock {\em Proc. Japan Acad. }, 64(A): 102 – 105, 
	1988. 
	
	
	\bibitem{BG1976b}
	M.~Baouendi and C.~Goulaouic.
	\newblock Singular nonlinear {C}auchy problems.
	\newblock {\em Journal of Differential Equations}, 22(2):268 – 291, 1976.
	
	\bibitem{BG1975}
	M.~S. Baouendi and C.~Goulaouic.
	\newblock Probl\`emes de {Cauchy} pseudo-diff\'erentiels analytiques.
	\newblock {\em Journ\'ees \'equations aux d\'eriv\'ees partielles}, pages
	27--41, 1975.
	
	\bibitem{BG1976a}
	M.~S. Baouendi and C.~Goulaouic.
	\newblock {Cauchy} problem for analytic pseudo-differential operators.
	\newblock {\em Communications in Partial Differential Equations},
	1(2):135--189, 1976.
	
	\bibitem{BB1977}
	C.~Bardos and S.~Benachour.
	\newblock Domaine d'analycit\'e des solutions de l'\'equation d'{E}uler dans un
	ouvert de {$\mathbb{R}^{n}$}.
	\newblock {\em Ann. Scuola Norm. Sup. Pisa Cl. Sci. (4)}, 4(4):647--687, 1977.
	
	\bibitem{BBZ1976}
	C.~Bardos, S.~Benachour, and M.~Zerner.
	\newblock Analyticité des solutions périodiques de l'équation d'{E}uler en
	deux dimensions.
	\newblock {\em C.R.A.S. Paris, Serie A}, 282:995 -- 998, 1976.
	
	\bibitem{BF1976}
	C.~Bardos and U.~Frisch.
	\newblock Finite-time regularity for bounded and unbounded ideal incompressible
	fluids using {H}older's estimates.
	\newblock In R.~Temam, editor, {\em Turbulence and Navier Stokes Equations},
	pages 1--13, Berlin, Heidelberg, 1976. Springer Berlin Heidelberg.
	
	\bibitem{Ben1979}
	S.~Benachour.
	\newblock Analyticité des solutions des équations d'{E}uler.
	\newblock {\em Archive for Rational Mechanics and Analysis}, 71(3):271--299,
	1979.
	
	\bibitem{BuB2012}
	M.~D. Bustamante and M.~Brachet.
	\newblock Interplay between the {Beale-Kato-Majda} theorem and the
	analyticity-strip method to investigate numerically the incompressible
	{E}uler singularity problem.
	\newblock {\em Physical Review E - Statistical, Nonlinear, and Soft Matter
		Physics}, 86(6), 2012.
	
	\bibitem{CLS20}
	R.E.~Caflisch, M.C,~Lombardo and M.~Sammartino. 
	\newblock{Vortex layers of small thickness.}
	\newblock{Commun. Pur. Appl. Math.}, 73(10): 2104--2179, 2020.
	
        \bibitem{CK18} 
        G.~Camliyurt and I.~Kukavica.
	\newblock On the {L}agrangian and {E}ulerian analyticity for the {E}uler equations.
	\newblock {\em Phys. D}, {376/377}:121--130, 2018.


	
	\bibitem{Del1985}
	J.-M. Delort.
	\newblock Estimations fines pour des operateurs pseudo-differentiels
	analytiques sur un ouvert a bord de ${R}^n$ application aux equations
	d’{E}uler.
	\newblock {\em Communications in Partial Differential Equations}, 10(12):1465--1525, 1985.
	
	\bibitem{FTZ2018}
	M.~Fei, T.~Tao, and Z.~Zhang.
	\newblock On the zero-viscosity limit of the {N}avier-{S}tokes equations in
	${R_+^3}$ without analyticity.
	\newblock {\em Journal des Mathematiques Pures et Appliquees}, 112:170--229,
	2018.
	
	\bibitem{FMB2003}
	U.~Frisch, T.~Matsumoto, and J.~Bec.
	\newblock Singularities of {E}uler flow? {N}ot out of the blue!
	\newblock {\em Journal of Statistical Physics}, 113(5-6):761--781, 2003.
	
	\bibitem{GSS11}
	F.~Gargano, M.~Sammartino, and V.~Sciacca. 
	\newblock High Reynolds number Navier-Stokes solutions and boundary layer
	separation induced by a rectilinear vortex. 
	\newblock {\em Computers \& Fluids}, 52(1): 73--91, 2011. 
	
	
	\bibitem{GSSC14a}
		F.~Gargano, M.~Sammartino,  V.~Sciacca, and K.W.~Cassel. 
	\newblock 	Analysis of complex singularities	in high-{R}eynolds-number {N}avier-{S}tokes solutions.
	\newblock {\em J.~Fluid Mech.}, 747: 381--421, 2014. 
	
	
	\bibitem{KSP2022}
	S.~S.~V. Kolluru, P.~Sharma, and R.~Pandit.
	\newblock Insights from a pseudospectral study of a potentially singular
	solution of the three-dimensional axisymmetric incompressible {E}uler
	equation.
	\newblock {\em Physical Review E}, 105(6), 2022.
	
	
	\bibitem{KNVW} 
	I.~Kukavica, T.T.~Nguyen, V.~Vicol, and F.~Wang.
	\newblock On the   {E}uler+{P}randtl expansion for the {N}avier-{S}tokes equations.
	\newblock {\em J.~Math.\ Fluid Mech.}, {24}(2): Paper No. 47, 46, 2022.

\bibitem{KWZ}       
I.~Kukavica, F.~Wang, and Y.~Zhu,
\newblock On Green's function of the vorticity formulation for the 3D Navier-Stokes equations
\newblock (submitted), 	arXiv:2407.10751.


	\bibitem{KTVZ11}
	I.~Kukavica, R.~Temam, V.C.~Vicol and  M.~Ziane.
	\newblock Local existence and uniqueness for the hydrostatic {E}uler
	equations on a bounded domain.
	\newblock {J.~Differ. Equations}, 250(3), 1719--1746, 2011.
	
	
	\bibitem{KV2009}
	I.~Kukavica and V.~Vicol.
	\newblock On the radius of analyticity of solutions to the three-dimensional
	{E}uler equations.
	\newblock {\em Proceedings of the American Mathematical Society},
	137(2):669--677, 2009.
	
	\bibitem{KV2011a}
	I.~Kukavica and V.~Vicol.
	\newblock On the analyticity and {G}evrey-class regularity up to the boundary
	for the {E}uler equations.
	\newblock {\em Nonlinearity}, 24(3):765--796, MAR 2011.
	
	\bibitem{KV2011b}
	I.~Kukavica and V.~Vicol.
	\newblock The domain of analyticity of solutions to the three-dimensional
	{E}uler equations in a half space.
	\newblock {\em Discrete and Continuous Dynamical Systems}, 29(1):285--303, JAN
	2011.

  
	\bibitem{KVW2020}
	I.~Kukavica, V.~Vicol, and F.~Wang.
	\newblock The inviscid limit for the {N}avier-{S}tokes equations with data
	analytic only near the boundary.
	\newblock {\em Archive for Rational Mechanics and Analysis}, 237(2):779--827,
	2020.
	
	\bibitem{LeB1986}
	D.~Le~Bail.
	\newblock Analyticité locale pour les solutions de l'equation d'{E}uler.
	\newblock {\em Archive for Rational Mechanics and Analysis}, 95(2):117--136,
	1986.
	
	\bibitem{LO1997}
	C.~D. Levermore and M.~Oliver.
	\newblock Analyticity of solutions for a generalized {E}uler equation.
	\newblock {\em Journal of Differential Equations}, 133(2):321--339, 1997.
	
	\bibitem{LS2001}
	C.~D. Levermore and M.~Sammartino.
	\newblock A shallow water model with eddy viscosity for basins with varying
	bottom topography.
	\newblock {\em Nonlinearity}, 14(6):1493--1515, 2001.
	
	\bibitem{LSZ2020}
	J.~Liao, F.~Sueur and P.~Zhang,
	\newblock Zero-viscosity limit of the incompressible Navier-Stokes equations with sharp vorticity gradient,
	\newblock {\em Journal of Differential Equations}, 268(10), 5855--5891, 2020.
	
	
	\bibitem{LCS04}
	M.C.~Lombardo, M.~Cannone and M.~Sammartino,
	\newblock Well-posedness of the boundary layer equations. 
	\newblock{\em Siam. J.~Math. Anal.}, 35(4): 987--1004, 2004.
	
	
	\bibitem{LH2014}
	G.~Luo and T.~Y. Hou.
	\newblock Potentially singular solutions of the 3{D} axisymmetric {E}uler
	equations.
	\newblock {\em Proceedings of the National Academy of Sciences of the United
		States of America}, 111(36):12968--12973, 2014.
	
	\bibitem{LH2019}
	G.~Luo and T.~Y. Hou.
	\newblock Formation of finite-time singularities in the 3{D axisymmetric E}uler
	equations: {A} numerics guided study.
	\newblock {\em SIAM Review}, 61(4):793--835, 2019.
	
	\bibitem{Mae2014}
	Y.~Maekawa.
	\newblock On the inviscid limit problem of the vorticity equations for viscous
	incompressible flows in the half-plane.
	\newblock {\em Communications on Pure and Applied Mathematics},
	67(7):1045--1128, 2014.
		
	\bibitem{MLB2015}
	R.~M. Mulungye, D.~Lucas, and M.~D. Bustamante.
	\newblock Symmetry-plane model of 3{D E}uler flows and mapping to regular
	systems to improve blowup assessment using numerical and analytical
	solutions.
	\newblock {\em Journal of Fluid Mechanics}, 771:468--502, 2015.
	
	\bibitem{NN2018}
	T.~Nguyen and T.~Nguyen.
	\newblock The inviscid limit of {N}avier-{S}tokes equations for analytic data
	on the half-space.
	\newblock {\em Archive for Rational Mechanics and Analysis}, 230(3):1103--1129,
	2018.
	
	\bibitem{Nir72}
	L.~Nirenberg.
	\newblock An abstract form of the nonlinear {Cauchy-Kowalewski} theorem. 
	\newblock {\em J.~Differ. Geom.} 6(4), 561--576, 1972.
	
	\bibitem{Saf1995}
	M.~V. Safonov.
	\newblock The abstract {C}auchy-{K}ovalevskaya theorem in a weighted {B}anach
	space.
	\newblock {\em Comm. Pure Appl. Math.}, 48(6):629--637, 1995.
	
	
	\bibitem{SC1998a}
	M.~Sammartino and R.~Caflisch.
	\newblock Zero viscosity limit for analytic solutions of the {N}avier- {S}tokes
	equation on a half-space {I}. {E}xistence for {E}uler and {P}randtl
	equations.
	\newblock {\em Communications in Mathematical Physics}, 192:433--461, 02 1998.
	
	\bibitem{SC1998b}
	M.~Sammartino and R.~Caflisch.
	\newblock Zero viscosity limit for analytic solutions of the {N}avier-{S}tokes
	equation on a half-space {II}. {C}onstruction of the {N}avier-{S}tokes
	solution.
	\newblock {\em Communications in Mathematical Physics}, 192:463--491, 02 1998.
		
	\bibitem{WWZ2017}
	C.~Wang, Y.~Wang, and Z.~Zhang.
	\newblock Zero-viscosity limit of the {N}avier-{S}tokes equations in the
	analytic setting.
	\newblock {\em Archive for Rational Mechanics and Analysis}, 224(2):555--595,
	2017.
	
	\bibitem{Wan2020}
	F.~Wang.
	\newblock The three-dimensional inviscid limit problem with data analytic near
	the boundary.
	\newblock {\em SIAM Journal on Mathematical Analysis}, 52(4):3520--3545, 2020.
	
\end{thebibliography}
\end{document}